\documentclass[final,leqno]{siamltex704}
\usepackage{amsmath}
\usepackage{graphicx}
\usepackage[notcite,notref]{showkeys}
\usepackage{xcolor}
\usepackage{wrapfig} 
\usepackage{float}
\usepackage{amsfonts,amssymb}
\usepackage{dsfont}
\usepackage{pifont}
\usepackage{hyperref}
\numberwithin{equation}{section}
\def\3bar{{|\hspace{-.02in}|\hspace{-.02in}|}}
\def\E{{\mathcal{E}}}
\def\T{{\mathcal{T}}}
\def\Q{{\mathcal{Q}}}

\def\pT{{\partial T}}

\def\LL{{\mathcal{L}}}

\def\bu{{\mathbf{u}}}
\def\bv{{\mathbf{v}}}
\def\bn{{\mathbf{n}}}
\def\be{{\mathbf{e}}}

\newtheorem{remark}{Remark}[section]
\newtheorem{algorithm}{Algorithm}[section]

\title{A Modified Primal-Dual Weak Galerkin Finite Element Method for Second Order Elliptic Equations in Non-Divergence Form}

\author{
Chunmei Wang\thanks{Department of Mathematics \& Statistics, Texas Tech University, Lubbock, TX 79409, USA (chunmei.wang@ttu.edu). The research of Chunmei Wang was partially supported by National Science Foundation Award DMS-1849483.} 
}

\begin{document}

\maketitle

\begin{abstract}
A modified primal-dual weak Galerkin (M-PDWG) finite element method is designed for the second order elliptic equation in non-divergence form. Compared with the existing PDWG methods proposed in \cite{wwnondiv}, the system of equations resulting from the M-PDWG scheme could be equivalently simplified into one equation involving only the primal variable by eliminating the dual variable (Lagrange multiplier). The resulting simplified system thus has significantly fewer degrees of freedom than the one resulting from existing PDWG scheme. In addition, the condition number of the simplified system could be greatly reduced when a newly introduced bilinear term in the M-PDWG scheme is appropriately chosen. Optimal order error estimates are derived for the numerical  approximations in the discrete $H^2$-norm, $H^1$-norm and $L^2$-norm respectively. 
 Extensive numerical results are demonstrated for both the smooth and non-smooth coefficients on convex and non-convex domains to verify the accuracy of the theory developed in this paper.
\end{abstract}

\begin{keywords} primal-dual, weak Galerkin, finite element methods, non-divergence form, Cord\`es condition, polyhedral meshes.
\end{keywords}

\begin{AMS}
65N30, 65N12, 35J15, 35D35
\end{AMS}

\pagestyle{myheadings}
\section{Introduction}
In this paper, we consider the second order elliptic equation in non-divergence form which seeks an unknown function $u=u(x)$ such that
\begin{equation}\label{1}
\begin{split}
\sum_{i, j=1}^d a_{ij}\partial^2_{ij}u&=f,\quad \text{in}\
\Omega,\\
u &=0,\quad \text{on}\ \partial\Omega,
\end{split}
\end{equation}
where $\Omega \subset \mathbb R^d(d=2,3)$ is an open bounded domain with Lipschitz continuous boundary $\partial\Omega$, the load function $f\in L^2(\Omega)$, and the coefficient tensor $a=(a_{ij})_{d\times d}\in [L^\infty(\Omega)]^{d\times d}$ is
symmetric, uniformly bounded and positive definite in the sense that there exist constants $C_1>0$ and $C_2>0$ such that
\begin{equation}\label{matrix}
C_1\xi^T\xi\leq \xi^T a \xi\le \, C_2\xi^T\xi, \qquad \forall
\xi\in \mathbb{R}^d,\ x\in\Omega.
\end{equation}
For the simplicity of notation, denote by $\LL:=\sum_{i, j=1}^d a_{ij}\partial_{ij}^2$ the second order
partial differential operator.

The second order elliptic problem in non-divergence form arises in various applications such as probability and stochastic processes \cite{Fleming}. This type of problem also plays an important role in the research of fully nonlinear partial differential equations in conjunction with linearization techniques (e.g., the Newton's iterative method) \cite{brenner-0, neilan}. In such applications, the coefficient tensor $a(x)$ is often hardly smooth. 
Therefore, it is crucial to develop effective numerical methods for the model problem (\ref{1}) with nonsmooth coefficient tensor. Readers are referred to \cite{wwnondiv} for more details of recent work developed for the model problem (\ref{1}) .

The goal of this paper is to develop a modified primal-dual weak Galerkin (M-PDWG) scheme for the second order elliptic problem in nondivergence form (\ref{1}), which is totally different from and  advantageous over the one proposed in \cite{wwnondiv}. 
The system of equations arising from the M-PDWG scheme could be equivalently simplified into one equation by eliminating its dual variable (Lagrange multiplier). The simplified system involves only the primal variable and thus has significantly fewer degrees of freedom compared to the PDWG scheme proposed in \cite{wwnondiv}. The main contributions of the present paper are (1)   the condition number of the simplified system could be significantly reduced when the $c(\cdot, \cdot)$ term is appropriately chosen; (2) the computational complexity of the simplified system is greatly reduced. Our theory for the M-PDWG method is based on two assumptions: (1) the $H^2$-regularity of the exact solution of the model problem (\ref{1}); and (2) the coefficient tensor $a(x)$ is piecewise continuous and satisfies the uniform ellipticity condition (\ref{matrix}). Optimal order error estimates are established for the primal variable in a discrete $H^2$-norm and for the dual variable in the $L^2$-norm. Moreover, the convergence theory is derived for the primal variable in the $H^1$ norm and $L^2$ norm under some smoothness assumptions for the coefficient tensor $a(x)$. Numerical examples are presented to illustrate the accuracy of the theory developed for the M-PDWG method.

The paper is organized as follows. In Section \ref{Section:Preliminaries}, we present the weak formulation for the model problem (\ref{1}). Section \ref{Section:Hessian} is devoted to a review of weak second order differential operator and its discretization. In Section \ref{Section:PD-WGFEM}, we describe the  M-PDWG finite element method for the model
problem (\ref{1}). Section 5 presents a simplified system resulting form the M-PDWG method proposed in Section 4.
Section \ref{Section:ExistenceUniquess} is devoted to a stability analysis for the M-PDWG scheme. Section 7 presents the error equations for the numerical scheme. In Section \ref{Section:ErrorEstimates}, we derive an optimal order
error estimate for the M-PDWG method in a discrete $H^2$ norm. Section \ref{Section:H1L2Error} establishes some error estimates in the usual $H^1$ norm and $L^2$ norm for the primal variable. In Section \ref{Section:NE}, the numerical experiments are presented for the M-PDWG scheme for smooth and non-smooth coefficient tensor $a(x)$ on convex and non-convex domains.

\section{Variational Formulations}\label{Section:Preliminaries}
We shall briefly review the weak formulation of the second order elliptic model problem  (\ref{1}) in non-divergence form \cite{wwnondiv}.

\begin{theorem}\label{THM:H2regularity} \cite{smears} Assume (1) $\Omega\subset \mathbb R^d$ is a bounded convex domain; (2) the coefficient tensor $a=(a_{ij}) \in [L^\infty(\Omega)]^{d\times d}$ satisfies the ellipticity condition (\ref{matrix}); and (3) the Cord\`es condition holds true; i.e., there exists an $\varepsilon\in (0,1]$ such that
\begin{equation}\label{cordes}
\frac{\sum_{i,j=1}^d a_{ij}^2}{(\sum_{i=1}^d a_{ii})^2} \leq
\frac{1}{d-1+\varepsilon}\qquad \mbox{in}\ \Omega.
\end{equation} 
There exists a unique strong solution $u\in H^2(\Omega)\cap H_0^1(\Omega)$ of the model problem (\ref{1}) satisfying
\begin{equation}\label{Assumption:H2regularity}
\|u\|_{2} \leq C \|f\|_0,
\end{equation}
for any given $f\in L^2(\Omega)$, where $C$ is a constant depending on $d$, the diameter of $\Omega$, $C_1$, $C_2$ and $\varepsilon$.
\end{theorem}

Throughout this paper, we assume the model problem (\ref{1}) has a
unique strong solution in $H^2(\Omega)\cap H_0^1(\Omega)$ with a priori estimate \eqref{Assumption:H2regularity}.
 
The variational formulation of the model problem (\ref{1}) seeks $u\in X=H^2(\Omega)\cap H_0^1(\Omega)$ such that
\begin{equation*}\label{Abstract-Formulation}
b(u, \sigma) = (f, \sigma)\qquad \forall \sigma\in Y=L^2(\Omega),
\end{equation*}
where
\begin{equation}\label{b-form}
b(u, \sigma) = (\LL u, \sigma).
\end{equation}
  The regularity assumption (\ref{Assumption:H2regularity}) implies that the bilinear form $b(\cdot,\cdot)$ satisfies the inf-sup condition
$$
\sup_{v\in X, v\neq 0} \frac{b(v,\sigma)}{\|v\|_X} \ge  \alpha
\|\sigma\|_Y,
$$
for all $\sigma\in Y$, where $\alpha$ is a generic constant related
to the constant $C$ in the $H^2$ regularity estimate
(\ref{Assumption:H2regularity}), $\|\cdot\|_X$ and  $\|\cdot\|_Y$ are the $H^2$ norm and the
$L^2$ norm, respectively.

\section{Discrete Weak Second Order Partial Derivative}\label{Section:Hessian} 
This section will briefly review the weak second order partial
derivative and its discrete version \cite{ww, wwnondiv}.

Let $T$ be a polygonal or polyhedral domain with boundary $\partial
T$. Denote by $v=\{v_0,v_b,\bv_g\}$ the weak function on the element $T$, where $v_0\in L^2(T)$ and $v_b\in
L^{2}(\partial T)$ are the values
of $v$ in the interior and on the boundary of $T$; and $\bv_g=(v_{g1},\ldots,v_{gd})\in [L^{2}(\partial T)]^d$ is the value of
 $\nabla v$ on the boundary of $T$. Note that
$v_b$ and $\bv_g$ may not necessarily be related to the traces of $v_0$
and $\nabla v_0$ on $\partial T$. It
is feasible to take $v_b$ as the trace of $v_0$ and leave
$\bv_g$ completely free or vice versa. 

Let $W(T)$ be the local space of the weak functions on $T$; i.e.,
\begin{equation}\label{2.1}
W(T)=\{v=\{v_0,v_b,\bv_g\}: v_0\in L^2(T), v_b\in L^{2}(\partial T),
\bv_g\in [L^{2}(\partial T)]^d\}.
\end{equation}
The weak second order partial derivative of the weak function $v\in W(T)$, denoted by $\partial^2_{ij,w} v$, is defined as a bounded linear functional on the Sobolev space $H^2(T)$ satisfying
\begin{equation}\label{2.3}
(\partial^2_{ij,w}v,\varphi)_T=(v_0,\partial^2_{ji}\varphi)_T-
 \langle v_b n_i,\partial_j\varphi\rangle_{\partial T}+
 \langle v_{gi},\varphi n_j\rangle_{\partial T},
 \end{equation}
 for any $\varphi\in H^2(T)$, where $\bn=(n_1,\cdots,n_d)$ is the unit outward normal direction on
$\partial T$.

Denote by $P_r(T)$ the space of polynomials with degree no more than $r\geq 0$ on $T$. A discrete version of $\partial^2_{ij,w} v$, denoted by $\partial^2_{ij,w,r,T} v$, is defined as the unique polynomial in $P_r(T)$ such that
\begin{equation}\label{2.4}
(\partial^2_{ij,w,r,T}v,\varphi)_T=(v_0,\partial^2
 _{ji}\varphi)_T-\langle v_b n_i,\partial_j\varphi\rangle_{\partial T}
 +\langle v_{gi},\varphi n_j\rangle_{\partial T},\quad \forall \varphi \in
 P_r(T).
\end{equation}
Applying the usual integration by parts to the first term on the
right-hand side of (\ref{2.4}) yields
\begin{equation}\label{2.4new}
(\partial^2_{ij,w,r,T}v, \varphi)_T=(\partial^2
 _{ij}v_0,\varphi)_T-\langle (v_b-v_0) n_i,\partial_j\varphi\rangle_{\partial T}
 +\langle v_{gi}-\partial_i v_0,\varphi n_j\rangle_{\partial T},
 \end{equation}
 for all $\varphi \in P_r(T)$, provided that $v_0\in H^2(T)$.

\section{Primal-Dual Weak Galerkin}\label{Section:PD-WGFEM}
Denote by ${\cal T}_h$ a finite element partition of the domain
$\Omega$ into polygons in 2D or polyhedra in 3D which is shape regular as described in \cite{wy3655}. Denote by
${\mathcal E}_h$ the set of all edges or flat faces in ${\cal T}_h$
and ${\mathcal E}_h^0={\mathcal E}_h \setminus \partial\Omega$ the
set of all interior edges or flat faces.  Denote by $h_T$ the diameter of the element $T\in {\cal T}_h$ and $h=\max_{T\in {\cal T}_h}h_T$ the meshsize of the partition ${\cal T}_h$. 

Let $k\geq 2$. Denote by $W_k(T)$ the local space of discrete weak functions; i.e., 
\begin{equation}\label{EQ:local-weak-fem-space}
W_k(T):=\{v=\{v_0,v_b,\bv_g\}\in P_k(T)\times P_k(e)\times
[P_{k-1}(e)]^d,\ e\in \partial T\cap\E_h\}.
\end{equation}
Patching $W_k(T)$ over all the elements $T\in {\cal T}_h$ through common value for $v_b$ on the interior interface $\E_h^0$  gives the weak finite element space; i.e.,
$$
W_{h,k}:=\big\{\{v_0,v_b, \textbf{v}_g\}:\ \{v_0,v_b, \bv_g\}|_T\in
W_k(T), \ T\in {\cal T}_h\big\}.
$$
Let $W_{h, k}^0$ be the subspace of $W_{h, k}$ with vanishing
boundary value for $v_b$ on $\partial\Omega$; i.e.,
\begin{equation*}\label{EQ:global-weak-fem-space}
W_{h,k}^0=\{\{v_0,v_b, \bv_g\}\in W_{h,k},\ v_b|_e=0, e\subset
\partial\Omega\}.
\end{equation*}

We further introduce the finite element space
\begin{equation*}\label{EQ:GlobalS}
V_{h, k}=\Big\{\sigma:\ \sigma|_T\in V_k(T),\ T\in {\cal T}_h\Big\},
\end{equation*}
where $V_k(T)$ is  chosen as either $P_{k-2}(T)$ or $P_{k-1}(T)$, as appropriate. The choice of $V_k(T)=P_{k-2}(T)$ has fewer degrees of freedom, while the choice $V_k(T)=P_{k-1}(T)$ results in more accurate  M-PDWG solution. 
 
For simplicity of notation, denote by $\partial^2_{ij, w}v$ the
discrete weak second order partial differential operator defined by
(\ref{2.4}) with $V_r(T)=V_k(T)$ on each element $T$; i.e.,
$$
(\partial^2_{ij, w} v)|_T=\partial^2_{ij,w,r,T}(v|_T), \qquad v\in
W_{h,k}.
$$

We introduce the bilinear forms
\begin{eqnarray}\label{EQ:global-b-form}
b_h(v, \sigma)&=&\sum_{T\in {\cal T}_h}b_T(v, \sigma),\quad v\in
W_{h,k}, \ \sigma\in V_{h,k},\\
s_h(u,v)&=&\sum_{T\in {\cal T}_h}s_T(u,v), \quad u,v\in
W_{h,k},\label{EQ:global-s-form}
\end{eqnarray}
where
\begin{eqnarray}\label{EQ:local-b-form}
b_T(v,\sigma)&=&\sum_{i,j=1}^d (a_{ij}\partial_{ij,w}^2
v,\sigma)_T,\\
s_T(u,v)&=&h_T^{-3}\langle u_0-u_b, v_0-v_b\rangle_\pT +
h_T^{-1}\langle \nabla u_0 -\textbf{u}_g, \nabla v_0
-\textbf{v}_g\rangle_\pT.\label{EQ:local-s-form}
\end{eqnarray}
We further introduce a symmetric and nonnegative continuous bilinear form
 $$
 c_h(\cdot, \cdot):V_{h,k}\times V_{h,k} \to  \mathbb R,
 $$
satisfying the continuity property; i.e., there exists a constant $C$ such that
\begin{equation}\label{norm}
 c_h(\lambda, \mu)\leq C\|\lambda\|_0 \| \mu\|_0,
 \end{equation}
 for any $\lambda, \mu\in V_{h, k}$, where $\|\cdot\|$ is the $L_2$ norm.
  
 \begin{algorithm}\emph{(M-PDWG Finite Element Method)}
\label{ALG:primal-dual-wg-fem} A  modified primal-dual weak Galerkin scheme for solving the second order elliptic problem (\ref{1}) in non-divergence form seeks $(u_h;\lambda_h)\in W_{h, k}^0 \times V_{h, k}$ satisfying
\begin{eqnarray}\label{2}
 s_h(u_h,v)+b_h(v, \lambda_h)&=&0,\qquad\qquad \forall v\in W_{h,k}^0,\\
 -c_h(\lambda_h, \sigma)+b_h(u_h, \sigma)&=&(f,\sigma),\qquad \forall\sigma\in V_{h,k}.\label{32}
\end{eqnarray}
Here $u_h$ is the primal variable and $\lambda_h$ is the dual variable or Lagrange multiplier.
\end{algorithm}

\section{Simplified M-PDWG Finite Element Methods} In order to greatly reduce the degrees of freedom and the computational complexity of the M-PDWG method \eqref{2}-\eqref{32}, we shall eliminate the dual variable $\lambda_h$ from the M-PDWG system resulting in a simplified system involving the primal variable $u_h$ only. 
   
Deonte by $\langle \cdot, \cdot \rangle$ the duality pairing between the  two spaces. For the bilinear forms $s_h(\cdot, \cdot)$, $b_h(\cdot, \cdot)$ and $c_h(\cdot, \cdot)$, we associate the operators $S\in {\cal L}(W_{h, k}^0; (W_{h, k}^0)')$, $B\in {\cal L}(W_{h,k}^0; V_{h, k}')$ and $C\in {\cal L}(V_{h, k}; V_{h, k}')$ defined by 
\begin{equation*}
\begin{split}
 \langle S u, v\rangle&=s_h(u, v), \qquad\forall u, v\in W_{h,k}^0, 
\\
 \langle B u, \mu\rangle&=b_h(u, \mu), \qquad  \forall u\in W_{h,k}^0, \mu\in V_{h,k},
 \\
 \langle C\lambda, \mu\rangle&=c_h(\lambda, \mu), \qquad \forall \lambda, \mu\in V_{h,k},
\end{split}
\end{equation*}
where we assume $c_h(\cdot, \cdot)$ is suitably constructed so that $C$ is invertible. As a specific example, for any $\rho, \sigma \in V_{h,k}$, we define 
\begin{equation}\label{ch}
c_h(\rho, \sigma)= 
 \sum_{T\in {\cal T}_h} h_T^2(\rho, \sigma)_T+h_T^3(\nabla\rho, \nabla\sigma)_T+ \sum_{i, j=1}^d h_T^4(\partial_{ij}^2 \rho, \partial_{ij}^2 \sigma)_T.       
\end{equation} 
Let $B' \in {\cal L}(V_{h, k}; (W_{h, k}^0)')$ be the dual operator of $B$; i.e.,
$$
  \langle B'\mu, u\rangle =\langle B u, \mu\rangle =b_h(u, \mu),\qquad \forall u\in W_{h,k}^0, \mu\in V_{h,k}.
$$
The M-PDWG scheme \eqref{2}-\eqref{32} can be equivalently rewritten as follows: Find $(u_h;\lambda_h)\in W_{h, k}^0 \times V_{h, k}$ satisfying 
 \begin{eqnarray}\label{eq1}
 Su_h+B' \lambda_h&=&0, \qquad \text{in} \quad (W_{h, k}^0)',\\
 -C \lambda_h +B u_h&=&f,  \qquad \text{in} \quad (V_{h, k})', \label{eq2}
\end{eqnarray}
where  $(W_{h, k}^0)'$ and $(V_{h, k})'$ are the dual spaces of $W_{h, k}^0$ and $V_{h, k}$, respectively.
Note that $C$ is invertible. Using \eqref{eq2}, we have
$$
\lambda_h=-C^{-1}(f-Bu_h),
$$
which, combined with (\ref{eq1}), leads to a simplified system as follows: Find $u_h\in W_{h, k}^0$, such that
\begin{equation}\label{sypli}
(S+B' C^{-1}B)u_h=B' C^{-1}f.
\end{equation}
 
Compared with the PDWG scheme for the second order elliptic problem in nondivergence form proposed in \cite{wwnondiv}, the  M-PDWG scheme is advantageous due to the facts that  (1) it could be reformulated into an equivalent simplified system (\ref{sypli}) involving the primal variable $u_h$ only; and (2) the condition number of \eqref{sypli} could be significantly reduced for a properly chosen $c(\cdot, \cdot)$ term. The main contributions of M-PDWG method can be generalized to PDWG methods for other model PDEs by adding an appropriately chosen $c(\cdot, \cdot)$ term.

\section{Stability and Solvability}\label{Section:ExistenceUniquess}
We shall demonstrate the existence and
uniqueness for the M-PDWG solution arising from Algorithm
\ref{ALG:primal-dual-wg-fem} through an inf-sup condition for the
bilinear form $b_h(\cdot, \cdot)$ . 

Let $k\geq 2$. On each element $T$, denote by $Q_0$ the $L^2$ projection onto $P_k(T)$. On each edge or face $e\subset\partial T$,
denote by $Q_b$ and $\textbf{Q}_g=(Q_{g1}, \ldots, Q_{gd})$
the $L^2$ projections onto $P_{k}(e)$ and $[P_{k-1}(e)]^d$,
respectively. For any function $w\in H^2(\Omega)$, denote by $Q_h w$ the $L^2$ projection onto the weak finite element space $W_{h,k}$ such
that on each element $T$, we have
\begin{equation}\label{EQ:OperatorQh}
Q_hw=\{Q_0w,Q_bw,\textbf{Q}_g(\nabla w)\}.
\end{equation}
Denote by ${\cal Q}_h$ the $L^2$ projection onto the space
$V_{h,k}$.

\begin{lemma}\label{Lemma5.1} \cite{ww} For any $w\in H^2(T)$, the commutative property holds true 
\begin{equation}\label{EQ:CommutativeP}
\partial^2_{ij,w}(Q_h w)={\cal Q}_h(\partial^2_{ij} w),\qquad
i,j=1,\ldots,d.
\end{equation}
\end{lemma}

We introduce the semi-norm for the weak finite element space $W_{h, k}$; i.e., 
\begin{equation}\label{EQ:triple-bar}
\| v\|^2_{2,h}=\sum_{T\in {\cal T}_h}  \|\sum_{i,j=1}^d {\cal Q}_h
(a_{ij}\partial_{ij}^2 v_0)\|_{T}^2 + s_h(v,v), \qquad\forall v\in W_{h,k}.
\end{equation}

\begin{lemma}\label{Lemma:Sept:01} \cite{wwnondiv} Assume that the coefficient matrix  $a=(a_{ij})$ is uniformly piecewise continuous in $\Omega$ with respect to the
finite element partition $\T_h$. There exists a fixed $h_0>0$ such
that if $v=\{v_0,v_b,\bv_g\}\in W_{h,k}^0$ satisfies
$\|v\|_{2,h}=0$, then we have $v\equiv 0$ for $h\le h_0$.
\end{lemma}

We further introduce another semi-norm for the weak finite element space $W_{h, k}$; i.e., for any $v\in W_{h,k}$, 
\begin{equation}\label{EQ:triple-bar-new}
\3bar v\3bar ^2_{2}=\sum_{T\in {\cal T}_h}  \|\sum_{i,j=1}^d {\cal
Q}_h (a_{ij}\partial_{ij,w}^2 v)\|_{T}^2 + s_h(v,v).
\end{equation}

The two semi-norms defined in (\ref{EQ:triple-bar}) and (\ref{EQ:triple-bar-new}) are equivalent, which is stated in the following lemma.

\begin{lemma}\label{Lemma:Sept:01-new}\cite{wwnondiv} Assume that the coefficient tensor $a=(a_{ij})$ is
uniformly piecewise continuous in $\Omega$ with respect to the
finite element partition $\T_h$. For any $v\in W_{h,k}$, there exist $\alpha_1>0$ and
$\alpha_2>0$ such that
\begin{equation*}\label{EQ:Sept:2015:001}
\alpha_1 \|v\|_{2,h} \leq \3bar v\3bar_{2} \leq \alpha_2 \|v\|_{2,h}.
\end{equation*}
 \end{lemma}

\begin{lemma}\label{lem3}\cite{wwnondiv}
(inf-sup condition) Assume that the coefficient tensor $a=(a_{ij})$ is uniformly piecewise continuous in $\Omega$ with respect to the finite element partition $\T_h$. For
any $\sigma\in V_{h,k}$, there exists $v_\sigma\in W_{h,k}^0$
satisfying
\begin{eqnarray} \label{EQ:August30:500}
b_h(v_\sigma,\sigma) & \ge & \frac12 \|\sigma\|_{0}^2,\\
\| v_\sigma\|^2_{2,h} &\leq & C \|\sigma\|^2_{0},
\label{EQ:August30:501}
\end{eqnarray}
provided that the meshsize $h<h_0$ for a sufficiently small, but
fixed parameter $h_0>0$.
\end{lemma}

\begin{theorem}\label{thmunique1} Assume that the coefficient matrix $a=(a_{ij})$
is uniformly piecewise smooth in $\Omega$ with respect to the
finite element partition $\T_h$. The M-PDWG finite element scheme (\ref{2})-(\ref{32}) has a unique solution
 $(u_h; \lambda_h)\in W_{h,k}^0 \times V_{h,k}$,
  provided that the
 meshsize $h<h_0$ holds true for a sufficiently small, but fixed
 parameter $h_0>0$. 
\end{theorem}
\begin{proof} It suffices to show that the homogeneous problem of (\ref{2})-(\ref{32}) has only the trivial solution. To this end, assume $f=0$. By choosing $v=u_h$ and $\sigma=\lambda_h$ in (\ref{2})-(\ref{32})  we arrive at
$$
s_h(u_h, u_h)+ c_h(\lambda_h, \lambda_h)=0,
$$
which implies $s_h(u_h, u_h)=0$ and $c_h(\lambda_h, \lambda_h)=0$. From $s_h(u_h, u_h)=0$, we have 
$u_0=u_b$ and $\nabla u_0=\bu_g$ on each $\partial T$, which gives $u_h\in C^1(\Omega)$. Therefore, from (\ref{2}), we have 
$$
b_h(v, \lambda_h)=0, \qquad \forall v\in W_{h,k}^0.
$$
From Lemma \ref{lem3}, for $\lambda_h\in V_{h, k}$, there exists $v_{\lambda_h}\in W_{h,k}^0$
satisfying
$$0=b_h(v_{\lambda_h}, \lambda_h)   \geq   \frac12 \|\lambda_h\|_{0}^2, 
 $$
 which gives $\lambda_h=0$ on each $T\in {\cal T}_h$ and further  $\lambda_h \equiv 0$ in $\Omega$. Substituting $\lambda_h \equiv 0$ in $\Omega$ into (\ref{32}) yields
\begin{equation}\label{bterm}
\begin{split}
0&=b_h(u_h, \sigma)\\
&=\sum_{T\in {\cal T}_h}\sum_{i, j=1}^d(a_{ij}\partial^2_{ij, w} u_h, \sigma)_T\\
&=\sum_{T\in {\cal T}_h}\sum_{i, j=1}^d( \partial^2_{ij, w} u_h, {\cal Q}_h(a_{ij}\sigma))_T\\
&=\sum_{T\in {\cal T}_h}\sum_{i, j=1}^d   (\partial^2
 _{ij}u_0,{\cal Q}_h(a_{ij}\sigma))_T-\langle (u_b-u_0) n_i,\partial_j {\cal Q}_h(a_{ij}\sigma)\rangle_{\partial T}
 \\&\quad+\langle u_{gi}-\partial_i u_0, {\cal Q}_h(a_{ij}\sigma) n_j\rangle_{\partial T}\\
 &=\sum_{T\in {\cal T}_h}\sum_{i, j=1}^d   (\partial^2
 _{ij}u_0,{\cal Q}_h(a_{ij}\sigma))_T, 
\end{split}
\end{equation}
for any $\sigma\in V_{h, k}$, where we used (\ref{2.4new}) together with  $u_0=u_b$ and $\nabla u_0=\bu_g$ on each $\partial T$. Letting ${\cal Q}_h(a_{ij}\sigma)=\partial^2
 _{ij}u_0$ in (\ref{bterm}) gives $\partial^2
 _{ij}u_0=0$ for any $i, j=1, \cdots, d$ on each element $T\in {\cal T}_h$.  Note that $u_0 \in  C^1(\Omega)$. Thus, we have $\Delta u_0=0$ in $\Omega$. Since $u_h\in W_{h,k}^0$, we have $u_0=u_b=0$ on $\partial \Omega$.   Therefore, $u_0\equiv 0$ in $\Omega$ and further $u_h\equiv 0$ in $\Omega$.
 
This completes the proof of the theorem.
\end{proof} 

\section{Error Equations} 
Let $(u_h;\lambda_h) \in W_{h,k}^0\times V_{h,k}$ be the M-PDWG solution arising from the numerical scheme (\ref{2})-(\ref{32}). Note that the dual problem $b(v, \lambda)=0$ has a trivial solution $\lambda=0$ for any $v\in H^2(\Omega)\cap H_0^1(\Omega)$. The error functions are respectively defined as follows
\begin{equation*}\label{error}
e_h=u_h-Q_hu,\quad \gamma_h =\lambda_h-\Q_h\lambda=\lambda_h.
\end{equation*}
  
\begin{lemma}\label{errorequa}
The following error equations for the M-PDWG scheme (\ref{2})-(\ref{32}) hold true; i.e.,
\begin{eqnarray}\label{sehv}
s_h(e_h, v)+b_h(v, \gamma_h) & = &
  -s_h(Q_h u,v),\qquad \forall v\in W_{h,k}^0,\\
-c_h(\gamma_h, \sigma)+b_h(e_h,\sigma) & = & \ell_u(\sigma), \qquad\;\; \qquad \forall
\sigma\in V_{h,k},\label{sehv2}
\end{eqnarray}
where
\begin{equation}\label{EQ:ell-u}
\ell_u(\sigma) = \sum_{T\in {\cal T}_h} \sum_{i,j=1}^d((I-{\cal
Q}_h)\partial_{ij}^2u, a_{ij}\sigma)_T.
\end{equation}
\end{lemma}

\begin{proof}
First, by subtracting $s_h(Q_h u,v)$ from both sides of (\ref{2}) we
obtain
\begin{align*}
s_h(u_h-Q_h u,v)+b_h(v, \lambda_h) = -s_h(Q_h u,v),\qquad\forall
v\in W_{h,k}^0,
\end{align*}
which implies 
\begin{equation*}\label{EQ:Error-EQ-01}
s_h(e_h,v)+b_h(v, \gamma_h) = -s_h(Q_h u,v),\qquad\forall v\in
W_{h,k}^0.
\end{equation*}
This completes the proof of the first error equation (\ref{sehv}).

To derive (\ref{sehv2}), we use (\ref{1}) and
  Lemma \ref{Lemma5.1} to obtain
\begin{equation*}
\begin{split}
b_h(Q_h u, \sigma) = &\sum_{T\in {\cal T}_h} (\sum_{i,j=1}^d
a_{ij}\partial_{ij,w}^2
Q_hu,\sigma)_T\\
=&\sum_{T\in {\cal T}_h} (\sum_{i,j=1}^d a_{ij}{\cal Q}_h\partial_{ij }^2
u,\sigma)_T\\
=&\sum_{T\in {\cal T}_h}(\sum_{i,j=1}^d a_{ij} \partial_{ij }^2 u,\sigma)_T+
\sum_{T\in {\cal T}_h} (\sum_{i,j=1}^da_{ij}({\cal Q}_h-I)\partial_{ij}^2 u, \sigma)_T\\
=& ( f,\sigma) +\sum_{T\in {\cal T}_h} \sum_{i,j=1}^d(({\cal
Q}_h-I)\partial_{ij}^2u, a_{ij}\sigma)_T,
\end{split}
\end{equation*}
for all $\sigma\in V_{h,k}$. Now subtracting the above equation from
(\ref{32}) yields the error equation (\ref{sehv2}). 

This completes the proof of the lemma.
\end{proof} 
 
\section{Error Estimates}\label{Section:ErrorEstimates}
Let $\T_h$ be a shape-regular finite element partition of
the domain $\Omega$. For any $T\in\T_h$, the
following trace inequality holds true \cite{wy3655}:
\begin{equation}\label{trace-inequality}
\|\varphi\|_{\pT}^2 \leq C
(h_T^{-1}\|\varphi\|_{T}^2+h_T\|\nabla\varphi\|_{T}^2),\qquad \forall \varphi\in H^1(T).
\end{equation}
Furthermore, assume $\varphi$ is a polynomial on the element $T\in \T_h$. Applying the inverse inequality  to (\ref{trace-inequality}) gives \cite{wy3655}
\begin{equation}\label{x}
\|\varphi\|_{\pT}^2 \leq C h_T^{-1}\|\varphi\|_{T}^2.
\end{equation}
 
\begin{lemma}\label{Lemma5.2}\cite{wy3655}  Assume that ${\cal T}_h$ is a shape regular finite element partition of the domain $\Omega$ as specified in \cite{wy3655}. For any $0\leq s\leq 2$
and $1\leq m\leq k$, there holds
\begin{eqnarray}\label{3.2}
\sum_{T\in {\cal T}_h}h_T^{2s}\|u-Q_0u\|^2_{s,T} &\leq &
Ch^{2(m+1)}\|u\|_{m+1}^2,\\
\label{3.3-3} \sum_{T\in {\cal T}_h}\sum_{i,j=1}^dh_T^{2s}\| u-{\cal
Q}_h u\|^2_{s,T} &\leq& Ch^{2(m-1)}\|u\|_{m-1}^2,\\
\label{3.3} \sum_{T\in {\cal
T}_h}\sum_{i,j=1}^dh_T^{2s}\|\partial^2_{ij}u-{\cal
Q}_h\partial^2_{ij}u\|^2_{s,T} &\leq& Ch^{2(m-1)}\|u\|_{m+1}^2.
\end{eqnarray}
\end{lemma}

We are ready to present the critical error estimates for the M-PDWG scheme (\ref{2})-(\ref{32}), which is the main contribution of this paper.  
\begin{theorem} \label{theoestimate}  Assume that the coefficient tensor $a=(a_{ij})$ is uniformly piecewise continuous in $\Omega$ with respect to the
finite element partition $\T_h$. Let $u$ be the exact solution of (\ref{1})  and $(u_h;\lambda_h) \in
W_{h,k}^0\times V_{h,k}$ be the M-PDWG solution of  (\ref{2})-(\ref{32}), respectively. Assume that the exact solution
$u$ of (\ref{1}) is sufficiently regular such that $u\in
H^{k+1}(\Omega)$. There exists a constant $C$ such that
 \begin{equation}\label{erres}
\| u_h-Q_h u \|_{2, h}+\|\lambda_h-{\cal Q}_h \lambda\|_0 \leq
Ch^{k-1}\|u\|_{k+1},
\end{equation}
provided that the meshsize $h<h_0$ holds true for a sufficiently
small, but fixed $h_0>0$.
\end{theorem}

\begin{proof}  
From \eqref{sehv}, we have
\begin{equation}\label{bt}
b_h(v, \gamma_h)=-s_h(Q_h u, v)-s_h(e_h, v).
\end{equation}
 
Recall that
\begin{equation}\label{EQ:September:03:100}
\begin{split}
s_h(Q_h u, v) = &\sum_{T\in {\cal T}_h}h_T^{-3}\langle Q_0u-Q_bu,
v_0-v_b\rangle_\pT\\
& +\sum_{T\in {\cal T}_h}h_T^{-1} \langle \nabla Q_0u
-\textbf{Q}_g(\nabla u), \nabla v_0 -\textbf{v}_g\rangle_\pT.
\end{split}
\end{equation}
The first term on the right-hand side of (\ref{EQ:September:03:100})
can be estimated by using the Cauchy-Schwarz inequality, the trace
inequality (\ref{trace-inequality}), and the estimate (\ref{3.2})
with $m=k$ as follows
\begin{equation}\label{s1}
\begin{split}
&\left| \sum_{T\in {\cal T}_h}h_T^{-3}\langle Q_0u-Q_bu, v_0-v_b\rangle_\pT\right|\\
= &\left| \sum_{T\in {\cal T}_h}h_T^{-3}\langle Q_0u- u, v_0-v_b\rangle_\pT\right|\\
 \leq & \Big(\sum_{T\in {\cal T}_h}h_T^{-3}\|u-Q_0u\|^2_{\partial
T}\Big)^{\frac{1}{2}} \Big(\sum_{T\in {\cal T}_h}
h_T^{-3}\|v_0-v_b\|^2_{\partial T}\Big)^{\frac{1}{2}}\\
\leq & C\Big(\sum_{T\in {\cal
T}_h}h_T^{-4}\big(\|u-Q_0u\|_T^2+h_T^2\|u-Q_0u\|_{1,T}^2\big)
\Big)^{\frac{1}{2}}(s_h(v, v))^{\frac{1}{2}}\\
\leq & Ch^{k-1}\|u\|_{k+1}(s_h(v, v))^{\frac{1}{2}}.
\end{split}
\end{equation}
Similarly, the second term on the right-hand side of
(\ref{EQ:September:03:100}) has the following estimate
\begin{equation}\label{s2}
\left|\sum_{T\in {\cal T}_h}h_T^{-1}\langle \nabla Q_0u
-\textbf{Q}_g(\nabla u ), \nabla v_0 -\textbf{v}_g \rangle_\pT
\right|\leq Ch^{k-1}\|u\|_{k+1} (s_h(v, v))^{\frac{1}{2}}.
\end{equation}
Combining (\ref{EQ:September:03:100}) - (\ref{s2})
gives 
\begin{equation}\label{EQ:True:01}
|s_h(Q_h u, v)| \leq C h^{k-1} \|u\|_{k+1} (s_h(v, v))^{\frac{1}{2}}.
\end{equation}

Using Cauchy-Schwarz inequality, it is easy to obtain 
\begin{equation}\label{EQ:True:01-1}
|s_h(e_h, v)| \leq  \big(s_h(e_h, e_h)\big)^{\frac{1}{2}}\big(s_h(v, v)\big)^{\frac{1}{2}}.
\end{equation}

Substituting (\ref{EQ:True:01})-(\ref{EQ:True:01-1}) into \eqref{bt} gives
$$
|b_h(v, \gamma_h) |\leq   (Ch^{k-1}\|u\|_{k+1}+(s_h(e_h, e_h))^{\frac{1}{2}}) (s_h(v, v))^{\frac{1}{2}},$$
which from Lemma \ref{lem3}, for $\gamma_h\in V_{h, k}$, there exists $v_{\gamma_h}\in W_{h, k}^0$ such that  
\begin{equation*}
\begin{split}
\frac{1}{2}\| \gamma_h \|_{0}^2\leq &|b_h(v_{\gamma_h}, \gamma_h) |\\\leq&   (Ch^{k-1}\|u\|_{k+1}+(s_h(e_h, e_h))^{\frac{1}{2}})\|v_{\gamma_h}\|_{2,h}\\\leq & (Ch^{k-1}\|u\|_{k+1}+(s_h(e_h, e_h))^{\frac{1}{2}})\| \gamma_h \|_{0}.
\end{split}
\end{equation*}
Therefore, we have
\begin{equation}\label{gama}
\| \gamma_h \|_{0}\leq  Ch^{k-1}\|u\|_{k+1}+(s_h(e_h, e_h))^{\frac{1}{2}}.
\end{equation}

From (\ref{sehv2}), we have
\begin{equation}\label{bht}
b_h(e_h, \sigma)=\ell_u(\sigma)+c_h(\gamma_h, \sigma).
\end{equation}

Using (\ref{EQ:ell-u}) and the estimate (\ref{3.3})  with $m=k$ we have
\begin{equation}\label{aij}
\begin{split}
|\ell_u(\sigma)| & = \left|\sum_{T\in\T_h} \sum_{i,j=1}^d (I-{\cal
Q}_h)\partial_{ij}^2u, a_{ij}\sigma)_T\right|\\
& \leq \sum_{i,j=1}^d \|a_{ij}\|_{L^\infty}\ \|(I-{\cal
Q}_h)\partial_{ij}^2u\|_0 \ \|\sigma\|_0 \\
&\leq C h^{k-1}\|u\|_{k+1} \|\sigma\|_0.
\end{split}
\end{equation}

Substituting (\ref{aij}) into (\ref{bht}), we have
\begin{equation*}
|b_h(e_h, \sigma)| \leq  C(h^{k-1}\|u\|_{k+1} +\|\gamma_h\|_0)\|\sigma\|_0, 
\end{equation*}
where we used (\ref{norm}). Taking $\sigma={\cal Q}_h (a_{ij}\partial_{ij, w}^2 e_h)$ in the above equation gives 
\begin{equation}\label{t1}
\Big(\sum_{T\in {\cal T}_h}\|\sum_{i, j=1}^d {\cal Q}_h(a_{ij} \partial_{ij, w}^2 e_h) \|_T^2\Big)^{\frac{1}{2}}\leq  C (h^{k-1}\|u\|_{k+1} +\|\gamma_h\|_0).
\end{equation} 

Letting $v=e_h$ in \eqref{sehv} and $\sigma=\gamma_h$ in \eqref{sehv2} gives
 \begin{equation}\label{ess2}
 s_h(e_h, e_h)+c_h(\gamma_h, \gamma_h)=-s_h(Q_hu, e_h)-\ell_u(\gamma_h).
 \end{equation}
Substituting  \eqref{EQ:True:01}, \eqref{gama} and \eqref{aij} into \eqref{ess2} yields
\begin{equation}\label{t2}
\begin{split}
&s_h(e_h, e_h)+c_h(\gamma_h, \gamma_h)\\
\leq &C h^{k-1}\|u\|_{k+1} ((s_h(e_h, e_h))^{\frac{1}{2}}+\|\gamma_h\|_0)
\\
 \leq &C h^{k-1}\|u\|_{k+1} ((s_h(e_h, e_h))^{\frac{1}{2}}+Ch^{k-1}\|u\|_{k+1}) \\
 \leq &C h^{2k-2}\|u\|^2_{k+1}+ C \frac{1}{\epsilon} h^{2k-2}\|u\|^2_{k+1}+ C\epsilon s_h(e_h, e_h) 
\end{split}
\end{equation}
where we used Young's inequality with $\epsilon$ being sufficiently small such that $1-C\epsilon >0$,
which gives 
\begin{equation*}\label{t3}
(1-C\epsilon)s_h(e_h, e_h) +c_h(\gamma_h, \gamma_h) \leq Ch^{2k-2}\|u\|_{k+1}^2,
\end{equation*}
which gives 
\begin{equation}\label{t4}
s_h(e_h, e_h) \leq Ch^{2k-2}\|u\|_{k+1}^2,
\end{equation}
where we used $c_h(\gamma_h, \gamma_h)$ is non-negative. Using \eqref{t4},  \eqref{gama} gives
\begin{equation}\label{t5}
 \|\gamma_h\|_0\leq Ch^{k-1}\|u\|_{k+1},
\end{equation}
 which, from (\ref{t1}) and (\ref{t4}), gives
\begin{equation}\label{t6}
 \3bar e_h\3bar_{2}\leq Ch^{k-1}\|u\|_{k+1}.
\end{equation}
Combining \eqref{t5} and \eqref{t6} and using Lemma \ref{Lemma:Sept:01-new} completes the proof of the theorem.

\end{proof} 

\section{Error Estimates in $H^1$ and
$L^2$}\label{Section:H1L2Error}
In this section, we shall establish the error estimates in $H^1$ and
$L^2$ norm for the M-PDWG solution arising from the scheme (\ref{2})-(\ref{32}).

\begin{lemma}\label{lemma7.2} \cite{wwnondiv} There exists a constant $C$ such that
for any $v\in W_{k}(T)$, we have
\begin{equation}\label{qaij}
\|\partial^2_{ij, w} v\|_T^2 \leq C\left(\|\partial_{ij}^2 v_0\|_T^2
+ s_T(v,v)\right).
\end{equation}
\end{lemma}
%

Consider an auxiliary problem: Find $w$ satisfying
\begin{align}\label{dual1}
\sum_{i,j=1}^d \partial_{ji}^2 (a_{ij} w)=& \ \theta,\qquad \text{in}\ \Omega,\\
w=& \ 0,\qquad \text{on}\ \partial\Omega, \label{dual2}
\end{align}
where $\theta$ is a given function.
The variational formulation for (\ref{dual1})-(\ref{dual2}) seeks $w\in
L^2(\Omega)$ such that \begin{equation}\label{Dual-Variational}
b(v, w) = (\theta, v), \qquad \forall v\in H^2(\Omega)\cap
H_0^1(\Omega),
\end{equation}
where the bilinear form $b(\cdot,\cdot)$ is given by (\ref{b-form}).

The problem (\ref{dual1})-(\ref{dual2}) is assumed to be $H^{1+s}$-regular ($s\in [0,1]$) in the sense that for any $\theta\in H^{s-1}(\Omega)$, there exists a unique $w\in H^{1+s}(\Omega)\cap H_0^1(\Omega)$ satisfying (\ref{Dual-Variational}) and a priori estimate:
\begin{equation}\label{regul}
\|w\|_{1+s}\leq C\|\theta\|_{s-1}.
\end{equation}

\begin{lemma}\label{Lemma:TechnicalEquality}\cite{wwnondiv}
Assume that the coefficient tensor $a=(a_{ij})\in  [C^1(\Omega)]^{d\times d}$. For any $v=\{v_0, v_b, \bv_g\}\in W_{h,k}^0$, there 
holds 
\begin{equation}\label{2.14:800}
\begin{split}
(v_0, \theta) =&\sum_{T\in{\cal T}_h} \sum_{i,j=1}^d
(a_{ij}\partial^2_{ij,w} v,
w)_T-\langle(v_{gi} -\partial_i v_0) n_j, ({\cal Q}_h-I)(a_{ij}w)\rangle_{\partial T}\\
&+ \langle (v_b-v_0) n_i,\partial_j( {\cal
Q}_h-I)(a_{ij}w)\rangle_{\partial T}.
\end{split}
\end{equation}
\end{lemma}

\begin{lemma}\label{Lemma:TechnicalEstimates:01}\cite{wwnondiv}
Assume that the coefficient matrix $a=(a_{ij}) \in [\Pi_{T\in\T_h} W^{1,\infty}(T)]^{d\times d}$. There exists a constant $C$ such that for any $v\in W_{h,k}^0$, we have
\begin{eqnarray}\label{2.14.100:10}
\left|\sum_{T\in{\cal T}_h} \sum_{i,j=1}^d  \langle(v_{gi}
-\partial_i v_0) n_j, ({\cal Q}_h-I)(a_{ij}w)\rangle_{\partial T}
\right| &\leq & Ch \ \|v\|_{2,h} \|\theta\|_{-1},
\\ \label{2.14.110:10} \left|\sum_{T\in{\cal T}_h} \sum_{i,j=1}^d
\langle(v_{b} -v_0) n_i,
\partial_j({\cal Q}_h-I)(a_{ij}w)\rangle_{\partial T}
\right| &\leq & Ch \  \|v\|_{2,h}\|\theta\|_{-1},
\end{eqnarray}
provided that the dual problem (\ref{Dual-Variational}) has the regularity estimate (\ref{regul}) with $s=0$.
\end{lemma}

\begin{lemma}\label{Lemma:TechnicalEstimates:02}
Assume that the coefficient matrix $a=(a_{ij}) \in \Pi_{T\in\T_h} [W^{2,\infty}(T)]^{d \times d}$ and $P_1(T)\subset V_k(T)$ for each element $T\in\T_h$. There exists a constant $C$ such that for any $v\in
W_{h,k}^0$, we have
\begin{eqnarray}\label{2.14.100:12}
\left|\sum_{T\in{\cal T}_h} \sum_{i,j=1}^d  \langle(v_{gi}
-\partial_i v_0) n_j, ({\cal Q}_h-I)(a_{ij}w)\rangle_{\partial T}
\right| & \leq & Ch^2 \ \|v\|_{2,h}\|\theta\|_{0},\\
\label{2.14.110:15} \left|\sum_{T\in{\cal T}_h} \sum_{i,j=1}^d
\langle(v_{b} -v_0) n_i,
\partial_j({\cal Q}_h-I)(a_{ij}w)\rangle_{\partial T}
\right|  &\leq & Ch^2 \  \|v\|_{2,h}\|\theta\|_{0},
\end{eqnarray}
provided that the regularity estimate (\ref{regul}) holds true with
$s=1$.
\end{lemma}
 
For convenience of analysis, in what follows of this paper, for any $\rho, \sigma \in V_{h, k}$, we shall employ the specific $c_h(\rho, \sigma)$ define in \eqref{ch}.
 
\begin{theorem}\label{Thm:H1errorestimate} 
Let $u_h=\{u_0, u_b, \bu_g\}\in W_{h, k}^0$ be the M-PDWG solution arising from the numerical scheme (\ref{2})-(\ref{32}). Assume that $a=(a_{ij})\in [C^1(\Omega)]^{d\times d}$ and the exact solution of the model problem (\ref{1}) is sufficiently regular such that $u\in H^{k+1}(\Omega)$. There exists a constant $C$ such that
\begin{equation}\label{e0-H1}
\left(\sum_{T\in\T_h}\|\nabla u_0 - \nabla u\|_T^2\right)^{\frac12}
\leq Ch^{k} \|u\|_{k+1},
\end{equation}
provided that the meshsize $h$ is sufficiently small and the dual
problem (\ref{dual1})-(\ref{dual2}) has $H^1$-regularity estimate (\ref{regul}) with $s=0$.
\end{theorem}

\begin{proof}
Given $\theta = -\nabla\cdot\eta$ with $\eta\in [C^1(\Omega)]^d$ satisfying $\eta=0$ on
$\E_h$, assume $w$ is the solution of the dual problem
(\ref{dual1})-(\ref{dual2}).  Taking $v=e_h$ in Lemma (\ref{Lemma:TechnicalEquality})  yields
\begin{equation}\label{2.14:800:10-1}
\begin{split}
-(e_0, \nabla\cdot\eta) =&\sum_{T\in{\cal T}_h} \sum_{i,j=1}^d
(a_{ij}\partial^2_{ij,w} e_h,
w)_T-\langle(e_{gi} -\partial_i e_0) n_j, ({\cal Q}_h-I)(a_{ij}w)\rangle_{\partial T}\\
&+ \langle (e_b-e_0) n_i,\partial_j( {\cal
Q}_h-I)(a_{ij}w)\rangle_{\partial T}\\
= & I_1 - I_2 + I_3,
\end{split}
\end{equation}
where $I_j(j=1, 2, 3)$ are defined accordingly. Due to $\eta=0$ on
$\E_h$, using the integration by parts to (\ref{2.14:800:10-1}) gives
\begin{equation}\label{2.14:800:10}
(\nabla e_0, \eta) = I_1 - I_2 + I_3.
\end{equation}
From Lemma \ref{Lemma:TechnicalEstimates:01} and $H^1$-regularity estimate (\ref{regul}) with $s=0$, the terms $I_2$ and $I_3$ are bounded as follows
\begin{equation}\label{2.14:800:15}
|I_2| + |I_3| \leq C h \|\theta\|_{-1} \|e_h\|_{2,h} \leq C h
\|\eta\|_0\|e_h\|_{2,h}.
\end{equation}
Regarding to the term $I_1$, from the error equation (\ref{sehv2}),  we have
\begin{equation}\label{2.14.120}
\begin{split}
I_1 = & \sum_{T\in{\cal T}_h} \sum_{i,j=1}^d
(a_{ij}\partial^2_{ij,w} e_h, w)_T\\
=& \sum_{T\in{\cal T}_h} \sum_{i,j=1}^d (a_{ij}\partial^2_{ij,w}
e_h, {\cal Q}_h w)_T + (a_{ij}\partial^2_{ij,w} e_h, (I-{\cal Q}_h)
w)_T\\
=& \sum_{T\in{\cal T}_h} \sum_{i,j=1}^d ((I-{\cal
Q}_h)\partial_{ij}^2u, a_{ij} {\cal Q}_h w)_T
+c_h(\gamma_h, {\cal Q}_h w)
\\&+ \sum_{T\in{\cal T}_h}
\sum_{i,j=1}^d(a_{ij}\partial^2_{ij,w} e_h, (I-{\cal Q}_h) w)_T\\
=&J_1+J_2+J_3,
\end{split}
\end{equation}
where $J_i$ for $i=1, 2, 3$ are defined accordingly. As to the term $J_1$, from Cauchy Schwarz inequality, we have
\begin{equation}\label{EQ:New:2015:800}
\begin{split}
|J_1|=& \Big|\sum_{T\in{\cal T}_h} ((I-{\cal Q}_h)\partial_{ij}^2u, a_{ij} {\cal Q}_h w)_T\Big| \\ 
=& \Big|\sum_{T\in{\cal T}_h}|((I-{\cal Q}_h)\partial_{ij}^2u, (I-{\cal Q}_h) a_{ij} {\cal Q}_hw)_T \Big|\\ 
\leq& \Big(\sum_{T\in{\cal T}_h} \|(I-{\cal Q}_h)\partial_{ij}^2u\|_T^2\Big)^{\frac{1}{2}} \Big( \sum_{T\in{\cal T}_h}  \|(I-{\cal Q}_h) a_{ij} {\cal Q}_hw\|_T^2 \Big)^{\frac{1}{2}}\\
 \leq & C h  \|(I-{\cal Q}_h)\partial_{ij}^2u\|  \|w\|_{1}.
\end{split}
\end{equation} 
 As to the term $J_2$, using Cauchy Schwarz inequality,  the inverse inequality and \eqref{ch} gives
 \begin{equation}\label{ch1}
 \begin{split}
|J_2|=&|c_h(\gamma_h, {\cal Q}_h w)| \\
\leq &\Big|\sum_{T\in {\cal T}_h} h_T^2(\gamma_h, {\cal Q}_h w)_T \Big|
+ \Big|\sum_{T\in {\cal T}_h} h_T^3(\nabla\gamma_h, \nabla {\cal Q}_h w)_T\Big|
 \\&+ \Big| \sum_{T\in {\cal T}_h} h_T^4\sum_{i, j=1}^d(\partial_{ij}^2 \gamma_h, \partial_{ij}^2 {\cal Q}_h w)_T\Big|
 \\
  \leq & Ch^2\|\gamma_h\|_0\|w\|_0+ Ch\|\gamma_h\|_0\|w\|_0+   Ch\|\gamma_h\|_0\|w\|_1   
\\ \leq &Ch\|\gamma_h\|_0\|w\|_1.
\end{split}
\end{equation} 
As to the term $J_3$, using Cauchy Schwarz inequality and (\ref{qaij}), we have
\begin{equation}\label{EQ:New:2015:810}
\begin{split}
|J_3|=&\Big|\sum_{T\in{\cal T}_h}
\sum_{i,j=1}^d(a_{ij}\partial^2_{ij,w} e_h,  (I-{\cal Q}_h) w)_T\Big|\\
=&\Big|\sum_{T\in{\cal T}_h}
\sum_{i,j=1}^d((a_{ij}-\bar{a}_{ij})\partial^2_{ij,w} e_h, (I-{\cal Q}_h) w)_T\Big|\\
\leq &  \Big(\sum_{T\in{\cal T}_h} \sum_{i,j=1}^d \|a_{ij}-\bar{a}_{ij}\|_{L^\infty(T)}^2 \|\partial^2_{ij,w}
e_h\|^2_T \Big)^{\frac{1}{2}}
 \Big(\sum_{T\in{\cal T}_h} 
 \|(I-{\cal Q}_h) w\|^2_T\Big)^{\frac{1}{2}}\\
\leq & C h  \|w\|_{1}\Big(\sum_{T\in{\cal T}_h}\sum_{i,j=1}^d (\varepsilon (h_T))^2(  \|\partial_{ij}^2
e_0\|_T^2+s_T(e_h,e_h))\Big)^{\frac12},
\end{split}
\end{equation}
where $\bar{a}_{ij}$ is the average of $a_{ij}$ on the element $T$ and $\varepsilon(h_T) \to 0$ as $h\to 0$. Substituting
(\ref{EQ:New:2015:800}) - (\ref{EQ:New:2015:810}) into (\ref{2.14.120}) yields
\begin{equation}\label{EQ:New:2015:820}
\begin{split}
&|I_1| \\
\leq & C h \Big(\varepsilon(h) \|\nabla^2 e_0\|_0 +\varepsilon(h)
\|e_h\|_{2,h}
+ \sum_{i,j=1}^d \|(I-{\cal Q}_h)\partial^2_{ij} u\|_0 +\|\gamma_h\|_0  \Big)\|w\|_1\\
\leq & C \Big(\varepsilon(h) \|\nabla e_0\|_0 + h\varepsilon(h) \|e_h\|_{2,h} + h
\sum_{i,j=1}^d \|(I-{\cal Q}_h)\partial_{ij}^2 u\|_0 +h\|\gamma_h\|_0 
\Big)\|\eta\|_0,
\end{split}
\end{equation}
where we used the inverse inequality and the estimate
$\|w\|_1\leq C \|\theta\|_{-1}\leq C \|\eta\|_0$. Substituting
(\ref{EQ:New:2015:820}) and (\ref{2.14:800:15}) into
(\ref{2.14:800:10}) gives
$$
|(\nabla e_0, \eta)| \leq C \Big(\varepsilon(h) \|\nabla e_0\|_0 +
h(1+\varepsilon(h) ) \|e_h\|_{2,h} + h \sum_{i,j=1}^d \|(I-{\cal Q}_h)\partial_{ij}^2
u\|_0 +h\|\gamma_h\|_0 \Big)\|\eta\|_0.
$$
Note that the set of all such $\eta$ is dense in $L^2(\Omega)$. The
above inequality implies
$$
\|\nabla e_0\|_0\leq C \Big(\varepsilon(h) \|\nabla e_0\|_0 + h(1+\varepsilon(h))
\|e_h\|_{2,h} + h \sum_{i,j=1}^d \|(I-{\cal Q}_h)\partial_{ij}^2
u\|_0 +h\|\gamma_h\|_0 \Big).
$$
Therefore, we have
\begin{equation}\label{EQf:New:2015:820:100}
\|\nabla e_0\|_0\leq  C h \Big(\|e_h\|_{2,h} + \sum_{i,j=1}^d
\|(I-{\cal Q}_h)\partial_{ij}^2 u\|_0 + \|\gamma_h\|_0  \Big)
\end{equation}
provided that the meshsize $h$ is sufficiently small such that $1-C\varepsilon(h)>0$ and $\varepsilon(h) \to 0$. The inequality
(\ref{EQf:New:2015:820:100}), the error estimate
(\ref{erres}), and the estimate (\ref{3.3}) with $m=k$ completes the proof of the
estimate (\ref{e0-H1}) using the usual triangle inequality and the estimate (\ref{3.2}) with $m=k$.
\end{proof}

We further present the $L^2$ error estimate for the primal variable $u_h$.

\begin{theorem}\label{Thm:L2errorestimate} Assume that (1) the
coefficients $a_{ij}  \in C^1(\Omega)\cap
\left[ \Pi_{T\in\T_h} W^{2,\infty}(T)\right]$ for $i, j=1, \cdots, d$; (2) the dual problem (\ref{dual1})-(\ref{dual2}) satisfies 
$H^2$-regularity estimate (\ref{regul}) with $s=1$; and (3) $P_1(T)\subset V_k(T)$ for any $T\in\T_h$. There exists a constant $C$ such that
\begin{equation}\label{e0}
\|u_0 - u\|_0 \leq Ch^{k+1} \|u\|_{k+1},
\end{equation}
provided that the meshsize $h$ is sufficiently small.
\end{theorem}

\begin{proof}  
Let $w$ be the solution of the dual problem (\ref{dual1})-(\ref{dual2}) for a given $\theta\in L^2(\Omega)$. Choosing $v=e_h$ in  Lemma \ref{Lemma:TechnicalEquality} yields
\begin{equation}\label{2.14:800:10:L2}
\begin{split}
(e_0, \theta) =&\sum_{T\in{\cal T}_h} \sum_{i,j=1}^d
(a_{ij}\partial^2_{ij,w} e_h,
w)_T-\langle(e_{gi} -\partial_i e_0) n_j, ({\cal Q}_h-I)(a_{ij}w)\rangle_{\partial T}\\
&+ \langle (e_b-e_0) n_i,\partial_j( {\cal
Q}_h-I)(a_{ij}w)\rangle_{\partial T}\\
= & J_1 - J_2 + J_3,
\end{split}
\end{equation}
where $J_i$ are defined accordingly for $i=1,2,3$. Using Lemma \ref{Lemma:TechnicalEstimates:02}, we obtain
\begin{equation}\label{2.14:800:15:L2}
|J_2| + |J_3| \leq C h^2 \|\theta\|_{0} \|e_h\|_{2,h}.
\end{equation}
As to the term $J_1$, using the error equation (\ref{sehv2})  gives rise to
\begin{equation}\label{2.14.120:L2}
\begin{split}
J_1 = & \sum_{T\in{\cal T}_h} \sum_{i,j=1}^d
(a_{ij}\partial^2_{ij,w} e_h, w)_T\\
=& \sum_{T\in{\cal T}_h} \sum_{i,j=1}^d (a_{ij}\partial^2_{ij,w}
e_h, {\cal Q}_h w)_T  + (a_{ij}\partial^2_{ij,w} e_h, (I-{\cal Q}_h)
w)_T\\
=& \sum_{T\in{\cal T}_h} \sum_{i,j=1}^d ((I-{\cal
Q}_h)\partial_{ij}^2u, a_{ij} {\cal Q}_h w)_T +c_h(\gamma_h, {\cal Q}_h w)\\
&+ \sum_{T\in{\cal T}_h}
\sum_{i,j=1}^d(a_{ij}\partial^2_{ij,w} e_h, (I-{\cal Q}_h) w)_T\\
=& I_1+I_2+I_3,
\end{split}
\end{equation}
where $I_i (i=1, 2, 3)$ are defined accordingly. Recall that $P_1(T)\subseteq V_k(T)$ and ${\cal Q}_h$ is the $L^2$
projection onto $V_k(T)$. As to the term $I_1$, using Cauchy-Schwarz inequality gives
\begin{equation}\label{EQ:New:2015:800:L2}
\begin{split}
|I_1|=&\Big|\sum_{T\in{\cal T}_h} \sum_{i,j=1}^d((I-{\cal Q}_h)\partial_{ij}^2u, a_{ij} {\cal Q}_h w)_T\Big| \\
=&
\Big|\sum_{T\in{\cal T}_h} \sum_{i,j=1}^d((I-{\cal Q}_h)\partial_{ij}^2u, (I-{\cal Q}_h) a_{ij} {\cal
Q}_hw)_T\Big|\\
\leq & \Big(\sum_{T\in{\cal T}_h} \sum_{i,j=1}^d\|(I-{\cal Q}_h)\partial_{ij}^2u\|^2_T\Big)^{\frac{1}{2}}\Big(\sum_{T\in{\cal T}_h} \sum_{i,j=1}^d\|(I-{\cal Q}_h) a_{ij}
{\cal Q}_hw\|_T^2\Big)^{\frac{1}{2}}\\
\leq & C h^2 \sum_{i,j=1}^d\|(I-{\cal Q}_h)\partial_{ij}^2u\|_0 \|w\|_{2}.
\end{split}
\end{equation}
  
As to the term $I_2$, using Cauchy-Schwarz inequality,  the inverse inequality and \eqref{ch} gives
\begin{equation}\label{c4}
\begin{split}
I_2=&\sum_{T\in {\cal T}_h}h_T^2(\gamma_h, {\cal Q}_h w)_T 
+ h_T^3(\nabla \gamma_h, \nabla {\cal Q}_h w)_T+\sum_{i, j=1}^d h_T^4(\partial_{ij}^2 \gamma_h, \partial_{ij}^2 {\cal Q}_h w)_T\\
\leq & Ch^2\Big(\sum_{T \in {\cal T}_h} \|\gamma_h\|_T^2\Big)^{\frac{1}{2}}\Big(\sum_{T \in {\cal T}_h}  \|{\cal Q}_h w\|_T^2\Big)^{\frac{1}{2}}\\
&+Ch^3\Big(\sum_{T \in {\cal T}_h} \|\nabla \gamma_h\|_T^2\Big)^{\frac{1}{2}}\Big(\sum_{T \in {\cal T}_h}  \|\nabla {\cal Q}_h w\|_T^2\Big)^{\frac{1}{2}}\\
&+Ch^4\Big(\sum_{T \in {\cal T}_h} \| \Delta \gamma_h\|_T^2\Big)^{\frac{1}{2}}\Big(\sum_{T \in {\cal T}_h}  \| \Delta {\cal Q}_h w\|_T^2\Big)^{\frac{1}{2}}\\
\leq &Ch^2  \|\gamma_h\|_0  \|w\|_2,
\end{split}
\end{equation} 

As to the term $I_3$, using (\ref{qaij}) yields
\begin{equation}\label{EQ:New:2015:810:L2}
\begin{split}
|I_3|=&|\sum_{T\in{\cal T}_h} \sum_{i,j=1}^d(a_{ij}\partial^2_{ij,w} e_h, (I-{\cal Q}_h) w)_T|\\
= &\ |\sum_{T\in{\cal T}_h} \sum_{i,j=1}^d((a_{ij}-\bar{a}_{ij})\partial^2_{ij,w} e_h, (I-{\cal Q}_h) w)_T|\\
\leq & \Big(\sum_{T \in {\cal T}_h} \sum_{i, j=1}^d \|a_{ij}-\bar{a}_{ij}\|_{L^\infty(T)}^2\|\partial^2_{ij,w}
e_h\|_T ^2\Big)^{\frac{1}{2}}  \Big(\sum_{T \in {\cal T}_h} \|(I-{\cal Q}_h) w\|_T^2\Big)^{\frac{1}{2}} \\
\leq & \ C h^3 \|w\|_{2} \Big(\sum_{T \in {\cal T}_h}\sum_{i, j=1}^d  \|\partial_{ij}^2
e_0\|_T^2+s_T(e_h,e_h)\Big)^{\frac{1}{2}},
\end{split}
\end{equation}
where $\bar{a}_{ij}$ is the average of $a_{ij}$ on the element $T\in {\cal T}_h$ such that $\|a_{ij}-\bar{a}_{ij}\|_{L^\infty(T)}\leq h_T$.

Using (\ref{EQ:New:2015:800:L2})-\eqref{EQ:New:2015:810:L2},  the inverse inequality and the regularity
assumption (\ref{regul}) for $s=1$, we have
\begin{equation}\label{EQ:New:2015:820:L2}
\begin{split}
|J_1| \leq & C ( h^3\|\nabla^2 e_0\|_0 + h^3 \|e_h\|_{2,h} +
h^2 \sum_{i,j=1}^d \|(I-{\cal Q}_h)\partial_{ij}^2 u\|_0 +h^2  \|\gamma_h\|_0
)\|w\|_2\\
\leq & C ( h^2\|\nabla e_0\|_0 + h^3 \|e_h\|_{2,h} + h^2
\sum_{i,j=1}^d \|(I-{\cal Q}_h)\partial_{ij}^2 u\|_0  +h^2  \|\gamma_h\|_0)\|\theta\|_0.
\end{split}
\end{equation}  

Substituting
(\ref{EQ:New:2015:820:L2}) and (\ref{2.14:800:15:L2}) into
(\ref{2.14:800:10:L2}) gives
$$
|(e_0, \theta)| \leq C h^2 \big( \|\nabla e_0\|_0 + \|e_h\|_{2,h} +
\sum_{i,j=1}^d \|(I-{\cal Q}_h)\partial_{ij}^2 u\|_0  +\|\gamma_h\|_0 
\big)\|\theta\|_0.
$$
This indicates
$$
\|e_0\|_0\leq C h^2 \big( \|\nabla e_0\|_0 + \|e_h\|_{2,h} +
\sum_{i,j=1}^d \|(I-{\cal Q}_h)\partial_{ij}^2 u\|_0 +\|\gamma_h\|_0
 \big),
$$
which, using (\ref{erres}),
(\ref{e0-H1}), (\ref{3.3}) with $m=k$, the usual triangle inequality and the estimate (\ref{3.2}) with $m=k$, completes the proof  of the theorem provided that the meshsize $h$ is sufficiently small.

\end{proof}

\begin{remark} \cite{wwnondiv} The optimal order error estimate (\ref{e0}) is based on the assumption that $P_1(T)\subseteq V_2(T)$, which is used to derive (\ref{2.14:800:15:L2}) and (\ref{EQ:New:2015:800:L2})-(\ref{EQ:New:2015:810:L2}). When it comes to the case of $P_1(T)\nsubseteq V_2(T)$, those inequalities are
modified by replacing $\|w\|_{2,T}$ by $h_T^{-1}\|w\|_{1,T}$. The conclusion is stated as follows: We assume (1) the coefficients $a_{ij}\in C^1(\Omega)$ for $i, j=1, \cdots, d$, (2) the meshsize $h$ is sufficiently small, and (3) the dual problem
(\ref{dual1})-(\ref{dual2}) satisfies  the $H^1$-regularity estimate (\ref{regul}) for $s=0$. The sub-optimal order error estimate holds true
\begin{equation*}\label{e0-L2-suboptimal}
\|u_0 - u\|_0 \leq Ch^{k} \|u\|_{k+1}.
\end{equation*}
\end{remark}
 
We introduce the following norms for the two boundary components
$u_b$ and $\bu_g$; i.e.,
\begin{equation*}\label{EQ:eb-eg-L2norm}
\|e_b\|_{0}:=\Big(\sum_{T\in {\cal T}_h} h_T\|e_b\|_{\partial
T}^2\Big)^{\frac{1}{2}},\quad \|\be_g\|_{0}:=\Big(\sum_{T\in {\cal
T}_h} h_T\|\be_g\|_{\partial T}^2\Big)^{\frac{1}{2}}.
\end{equation*}

\begin{theorem}\label{Thm:L2errorestimate-ub}\cite{wwnondiv}
Under the assumptions of Theorem \ref{Thm:L2errorestimate}, there
exists a constant $C$ such that
\begin{eqnarray*}\label{eb}
\|u_b- Q_b u\|_{0} &\leq & Ch^{k+1} \|u\|_{k+1},\\
\|\bu_g - {\bf Q}_b \nabla u\|_{0} &\leq & Ch^{k}
\|u\|_{k+1}.\label{eg}
\end{eqnarray*}
\end{theorem}
 
\section{Numerical Experiments}\label{Section:NE}
A series of the numerical results are illustrated to verify the accuracy of the theory developed for the M-PDWG method (\ref{2})-(\ref{32}). 
 
We shall take the lowest order WG element with $k=2$ on triangular
partitions as an example in the implementation. The finite element spaces are thus respectively given by
$$
W_{h,2}=\{v=\{v_0,v_b, \bv_g\}:\ v_0\in P_2(T), v_b\in P_2(e),
\bv_g\in [P_1(e)]^2, \forall T\in {\cal T}_h, e\in \E_h \},
$$ 
$$
V_{h,2}=\{\sigma: \ \sigma|_T \in V_2(T),\ \forall T\in {\cal T}_h
\},
$$
where both $V_2(T)=P_1(T)$ and $V_2(T)=P_0(T)$ are considered. A finite element function $v\in W_{h,2}$ is named $C^0$-type if $v_b= v_0|_\pT$ for each element $T$. The $C^0$-type WG element leads to a linear system with less computational complexity compared with the general WG elements. However, the $C^0$ continuity does not permit the availability of polygonal elements. Note that the theoretical results developed in this paper could be generalized to $C^0$-type triangular elements without any difficulty. The $C^0$-type WG element with $V_2(T)=P_1(T)$ is called the $P_2(T)/[P_1(\pT)]^2/P_1(T)$ element; and the $C^0$-type WG element with $V_2(T)=P_0(T)$ is called the $P_2(T)/[P_1(\pT)]^2/P_0(T)$ element.

Three domains are used in our numerical experiments: (1) the unit square domain $\Omega_1=(0,1)^2$; (2) the square domain $\Omega_2=(-1,1)^2$; and (3) the non-convex L-shaped domain $\Omega_3$ with vertices $A_0=(0,0), \ A_1=(2,0), \ A_2=(1,1), \ A_3=(1,2),$ and $A_4=(0,2)$. Starting from a given initial coarse triangulation of the domain, the triangular partition is obtained by successively dividing each coarse level triangle into four congruent sub-triangles through connecting the mid-points on each edge of each triangle.

Let $u_h=\{u_0, \bu_g\}\in W_{h,2}$ and $\lambda_h\in V_{h,2}$ 
be the M-PDWG solution arising from the scheme
(\ref{2})-(\ref{32}). Recall that the exact solution of Lagrange multiplier is $\lambda=0$. These numerical solutions are compared with the interpolants of the corresponding exact solutions; i.e.,
$$
e_h=\{e_0,\textbf{e}_g\}=\{u_0- I_h u, \ \bu_g - {\bf I}_g (\nabla
u)\},\quad \gamma_h=\lambda_h-0,
$$
where $I_h u$ is the Lagrange interpolation of the exact solution $u$ on each triangular element using three vertices and three mid-points on the edges, and ${\bf I}_g (\nabla u)$ is the linear interpolant of $\nabla u$ on each edge $e\in \E_h$.  
The following $L^2$ norms are employed to measure the errors:
\begin{eqnarray*}
 &&  \|e_0\|_0=\Big(\sum_{T\in {\cal T}_h}
\int_T |e_0|^2 dT\Big)^{\frac{1}{2}},\qquad 
\|\textbf{e}_g\|_0=\Big(\sum_{T\in {\cal T}_h} h_T
\int_{\partial T}
|\textbf{e}_g|^2 ds\Big)^{\frac{1}{2}}, \\
 &&  \|\gamma_h\|_0=\Big(\sum_{T\in {\cal
T}_h} \int_T |\gamma_h|^2 dT\Big)^{\frac{1}{2}}.
\end{eqnarray*}
 
\noindent {\bf Test Case 1.} Find $u$ such that
\begin{equation}\label{Test-Problem}
\begin{split}
\sum_{i, j=1}^2 a_{ij}\partial^2_{ij}u= &f,\quad \text{in}\
\Omega,\\
u = & g,\quad \text{on}\ \partial\Omega,
\end{split}
\end{equation}
where $\Omega=\Omega_i (i=1, 3)$, and the exact solution is given by $u=\sin(x_1)\sin(x_2)$.

Tables \ref{NE:TRI:Case2-1}-\ref{NE:TRI:Case2-2} show the
numerical results for the M-PDWG method \eqref{2}-\eqref{32} for the test problem (\ref{Test-Problem}) when the $C^0$-$P_2(T)/[P_1(\pT)]^2/P_1(T)$ element is applied. We observe from Tables \ref{NE:TRI:Case2-1}-\ref{NE:TRI:Case2-2} that the convergence rates for $e_0$ in the discrete $L^2$-norm are of orders ${\cal O}(h^4)$ and ${\cal O}(h^{3.6})$ on the unit square domain $\Omega_1$ and on the L-shaped domain $\Omega_3$, respectively. The convergence rates for $\be_g$ and $\gamma_h$ in the discrete $L^2$ norm are of orders ${\cal O}(h^2)$ and ${\cal O}(h)$ on both $\Omega_1$ and $\Omega_3$ respectively. Note that the expected optimal convergence rates for $e_0$, $\be_g$ and $\gamma_h$ in the discrete $L^2$-norm on the convex domain $\Omega_1$ are of orders ${\cal O}(h^3)$, ${\cal O}(h^2)$ and ${\cal O}(h)$, respectively. When it comes to the non-convex L-shaped domain $\Omega_3$, the theoretical order of convergence for $e_0$ in the discrete $L^2$-norm should be between ${\cal O}(h^2)$ and ${\cal O}(h^3)$ due to the lack of  $H^2$-regularity required for the dual problem (\ref{dual1})-(\ref{dual2}). However, the theoretical rates of convergence for $\be_g$ and $\gamma_h$ remain to be of orders ${\cal O}(h^2)$ and ${\cal O}(h)$, respectively. It is clear that the numerical results are greatly consistent with the theory for $\be_g$ and $\gamma_h$ in the discrete $L^2$-norm, and outperform the theory for $e_0$ in the discrete $L^2$-norm for the case of smooth solutions with smooth coefficients on uniform triangular partitions.

\begin{table}[h!]
\begin{center}
\caption{Test Case 1: 
 Convergence rates for $C^0$-$C^0$-$P_2(T)/[P_1(\pT)]^2/P_1(T)$ element on $\Omega_1$.}\label{NE:TRI:Case2-1}
\begin{tabular}{|c|c|c|c|c|c|c|}
\hline
$1/h$  & $\|e_0\|_0 $ & order &  $\|\be_g\|_{0}$  & order  &   $\|\gamma_h\|_0$  & order  \\
\hline
1 &0.006248	&&	0.1260 	&&	3.36E-04	&
\\
\hline
2& 0.001470 &	2.087 &	0.04477 &	1.493&	 6.51E-04& -0.9546
\\
\hline
4&1.39E-04&	3.399&	0.01157 &1.952 &2.84E-04	&1.195
\\
\hline
8&1.03E-05&	3.753 &	0.002843	&2.025 	&1.32E-04&	1.102
\\
\hline
16&6.97E-07&	3.891&	7.02E-04&	2.017 &	6.43E-05	&1.043
\\
\hline
32& 4.54E-08&	3.940 &	1.75E-04&	2.007 &	3.17E-05&	1.018  
\\
\hline
\end{tabular}
\end{center}
\end{table}

\begin{table}[h!]
\begin{center}
\caption{Test Case 1: 
Convergence rates for $C^0$-$P_2(T)/[P_1(\pT)]^2/P_1(T)$ element on $\Omega_3$.}\label{NE:TRI:Case2-2}
\begin{tabular}{|c|c|c|c|c|c|c|}
\hline
$1/h$        & $\|e_0\|_0 $ & order &  $\|\be_g\|_{0} $  & order  &   $\|\gamma_h\|_0$  & order  \\
\hline
1	& 0.01676	&&	0.4804 	&&	0.004498	&
\\
\hline
2	&0.002489 &	2.751&	0.1248&	1.945 &	0.001956&	1.201 
\\
\hline
4&	2.30E-04	&3.435 &	0.03100&	2.009&	8.76E-04&	1.160
\\
\hline
8&	1.94E-05&	3.572&	0.007674&	2.014 	&4.13E-04	&1.082 
\\
\hline
16&	1.61E-06	&3.585 &	0.001907 &	2.008 &	2.02E-04	&1.035
\\
\hline
32&	1.37E-07	&3.557 &	4.75E-04	&2.006&	9.99E-05&	1.015
\\
\hline
\end{tabular}
\end{center}
\end{table}

\noindent{\bf Test Case 2.}  Find $u$ such that
\begin{equation}\label{EQ:NE:500}
\begin{split}
\sum_{i,j=1}^2 (1+\delta_{ij}) \frac{x_i}{|x_i|}\frac{x_j}{|x_j|}
\partial^2_{ij} u & = f\qquad \mbox{in } \Omega,\\
u & = 0\qquad \mbox{on } \partial\Omega,
\end{split}
\end{equation}
where $\Omega_2=(-1,1)^2$, and the exact solution is 
$u= x_1 x_2 (1-e^{1-|x_1|}) (1-e^{1-|x_2|})$. It is easy to check the Cord\`es condition (\ref{cordes}) is
satisfied for the test problem (\ref{EQ:NE:500}) with $\varepsilon = 3/5$
and the coefficient matrix $a=(a_{ij})$ is discontinuous across the $x_i (i=1, 2)$ axis.  

Table \ref{NE:TRI:Case10-1} presents the numerical performance of the M-PDWG scheme \eqref{2}-\eqref{32} for the
test problem (\ref{EQ:NE:500}) when the $C^0$-$P_2(T)/[P_1(\pT)]^2/P_1(T)$ element is employed. The numerical results indicate that the convergence rate for $\be_g$ in the discrete $L^2$ norm is of an expected optimal order ${\cal O}(h^2)$. The convergence rate for the Lagrange multiplier in the discrete $L^2$ norm seems to be of an order higher than the expected order ${\cal O}(h)$. The convergence order for $e_0$ in the discrete $L^2$ norm seems to be of an order ${\cal O}(h^2)$. Note that it is not clear to us whether the dual problem (\ref{dual1})-(\ref{dual2}) has the regularity required for the convergence analysis. There are no theoretical results on the convergence rate for $e_0$ in the discrete $L^2$ norm. Table \ref{NE:TRI:Case10-2} shows the numerical results for the test problem (\ref{EQ:NE:500}) when the $C^0$-$P_2(T)/[P_1(\pT)]^2/P_0(T)$ element is applied. We observe from Table \ref{NE:TRI:Case10-2} that the convergence rates for $e_0$, $\be_g$ and $\gamma_h$ in the discrete $L^2$ norm seem to be a little higher than the convergence order corresponding to the case  when the $C^0$-$P_2(T)/[P_1(\pT)]^2/P_1(T)$ element is employed. 
 
\begin{table}[h!]
\begin{center}
\caption{Test Case 2: Convergence rates for $C^0$-
$P_2(T)/[P_1(\pT)]^2/P_1(T)$ element on $\Omega_2$.}\label{NE:TRI:Case10-1}
\begin{tabular}{|c|c|c|c|c|c|c|}
\hline
$2/h$        & $\|e_0\|_0 $ & order &  $\|\be_g \|_0 $  & order  &   $\|\gamma_h\|_0$  & order  \\
\hline 
1 &0.6160 	&&	2.554&&	1.000 &
\\
\hline
2&	0.4621&	0.4148 &	1.676 &	0.6074 	&0.8970	&0.1572
\\
\hline
4&	0.1389&	1.734 &	1.006&	0.7369&	3.270 	&-1.866 \\
\hline
8& 	0.02019 &	2.782 &	0.1339 &	2.909 &	0.6337&	2.368
\\
\hline
16&	0.006505 &	1.634 &	0.03229	&2.052 &0.2249	&1.494 \\
\hline
32&	0.001640	&1.988 	&0.007814	&2.047&0.09469 &	1.248
\\
\hline
\end{tabular}
\end{center}
\end{table}

Figures \ref{test1-1}-\ref{test1-2} illustrate the numerical error for the Lagrange multiplier $\lambda_h$ when the $C^0$-$P_2(T)/[P_1(\pT)]^2/P_1(T)$ element and the $C^0$-$P_2(T)/[P_1(\pT)]^2/P_0(T)$ element are employed respectively, compared with the PDWG scheme proposed in \cite{wwnondiv}.

\begin{figure}[h]
\centering
\begin{tabular}{cc}
\resizebox{2.2in}{2in}{\includegraphics{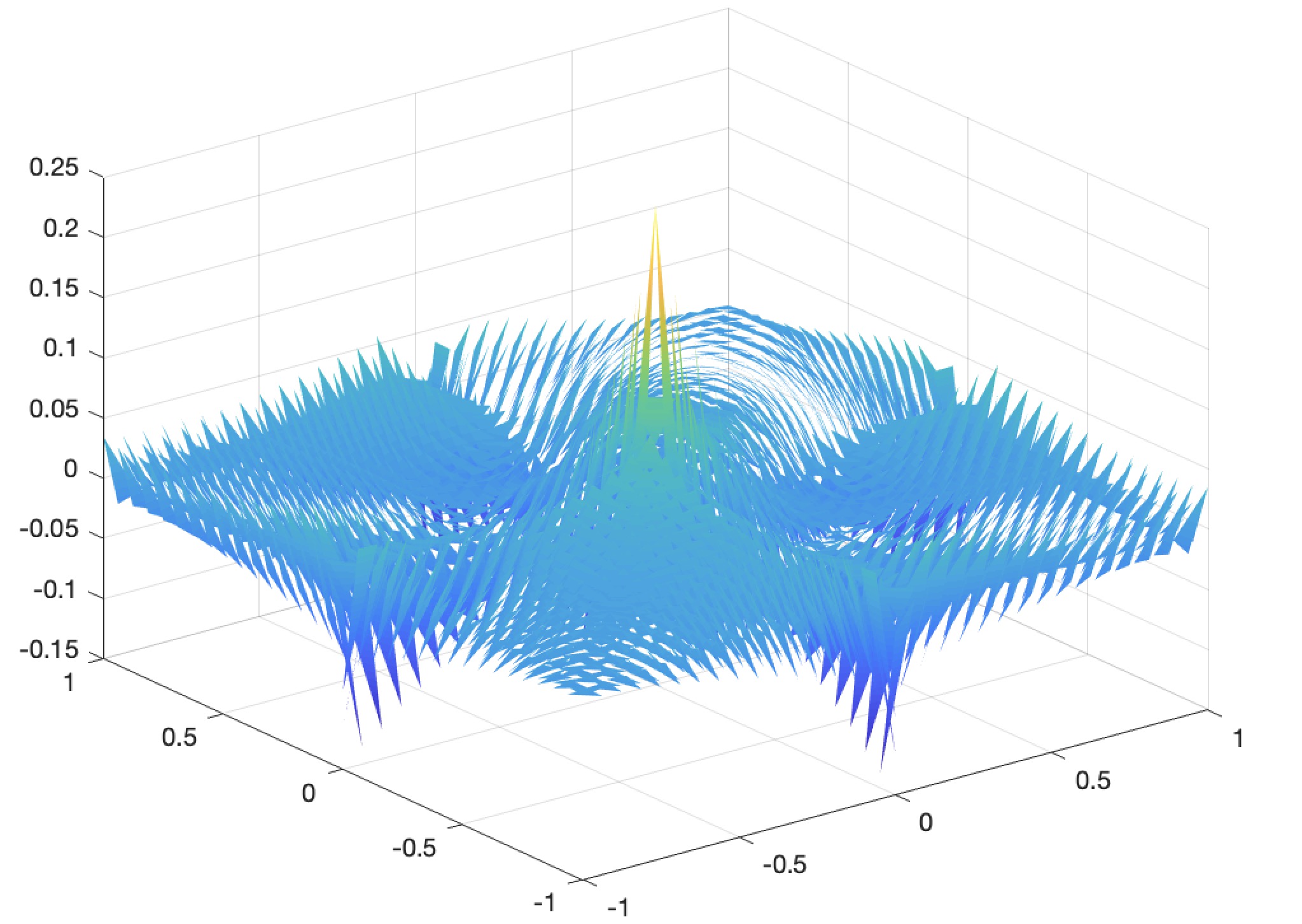}}
\resizebox{2.2in}{2in}{\includegraphics{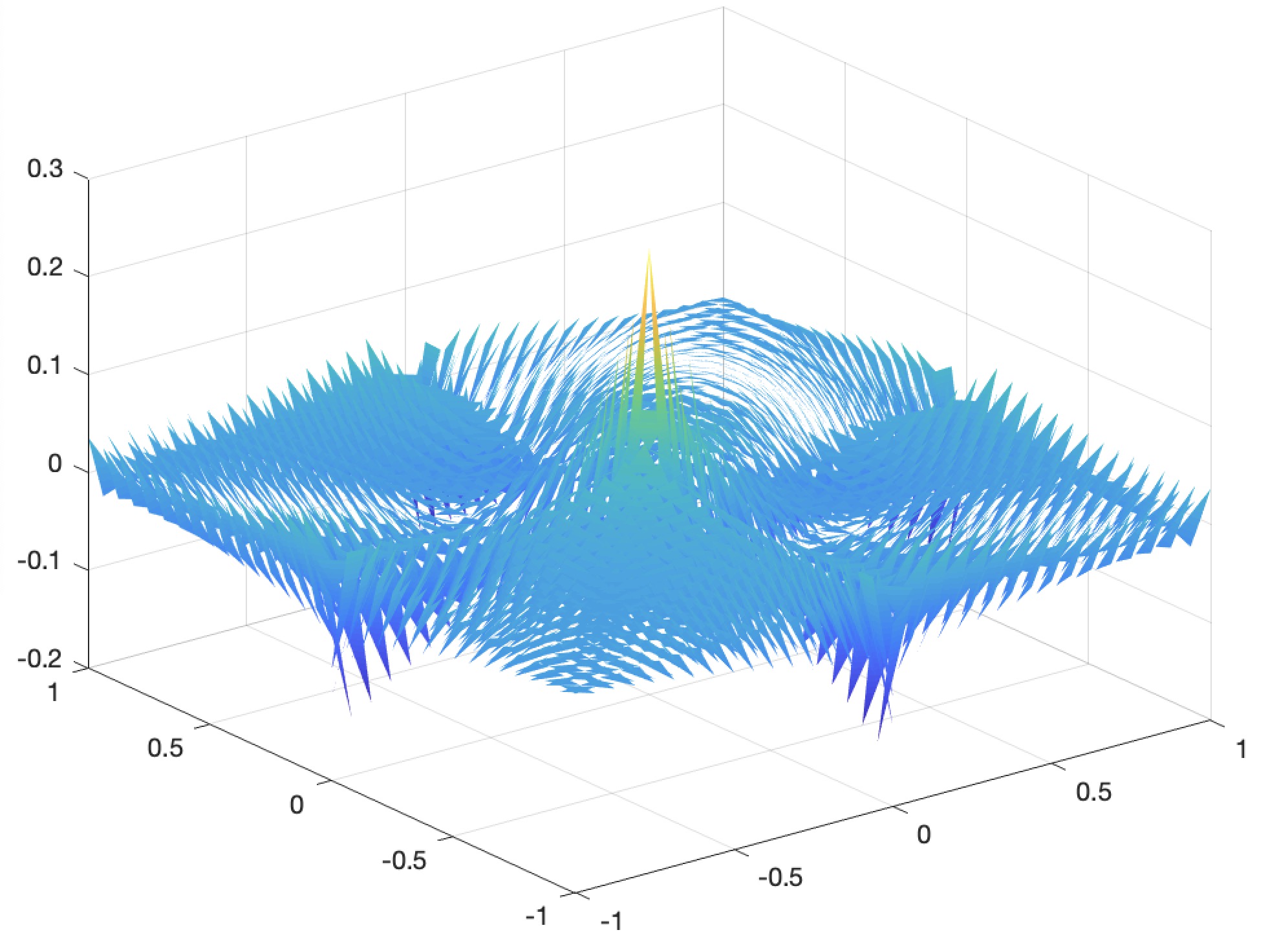}}
\end{tabular}
\caption{Test Case 2: Numerical error for Lagrange multiplier when $C^0$-$P_2(T)/[P_1(\pT)]^2/P_1(T)$ element is applied: left figure is without the term $c(\cdot, \cdot)$ proposed in \cite{wwnondiv}; right figure is with the term $c(\cdot, \cdot)$ proposed in this paper.}
\label{test1-1}
 \end{figure}

\begin{table}[h!]
\begin{center}
\caption{Test Case 2: Convergence rates for $C^0$-
$P_2(T)/[P_1(\pT)]^2/P_0(T)$ element on $\Omega_2$.}\label{NE:TRI:Case10-2}
\begin{tabular}{|c|c|c|c|c|c|c|}
\hline
$2/h$        & $\|e_0\|_0 $ & order &  $\|\be_g \|_0 $  & order  &   $\|\gamma_h\|_0$  & order  \\
\hline
1&	0.1590 	&&	0.7950 	&&	0.07950 	&
\\
\hline
2&	0.2253&	-0.5027&	1.383 &	-0.7982 &	0.3321	&-2.062 
\\
\hline
4&	0.1963 &	0.1984&	0.7627&	0.8582&	0.2444 &	0.4423
\\
\hline
8&	0.06727&	1.545&	0.2109 &	1.854 &	0.1349	&0.8577 
\\
\hline
16	&0.01536 &	2.130&	0.04616 &	2.192 	&0.05452 &1.307
\\
\hline
32&	0.003276 &	2.230 &	0.01020 &	2.178 	&0.02134	&1.354
\\
\hline
\end{tabular}
\end{center}
\end{table}

\begin{figure}[h]
\centering
\begin{tabular}{cc}
\resizebox{2.2in}{2in}{\includegraphics{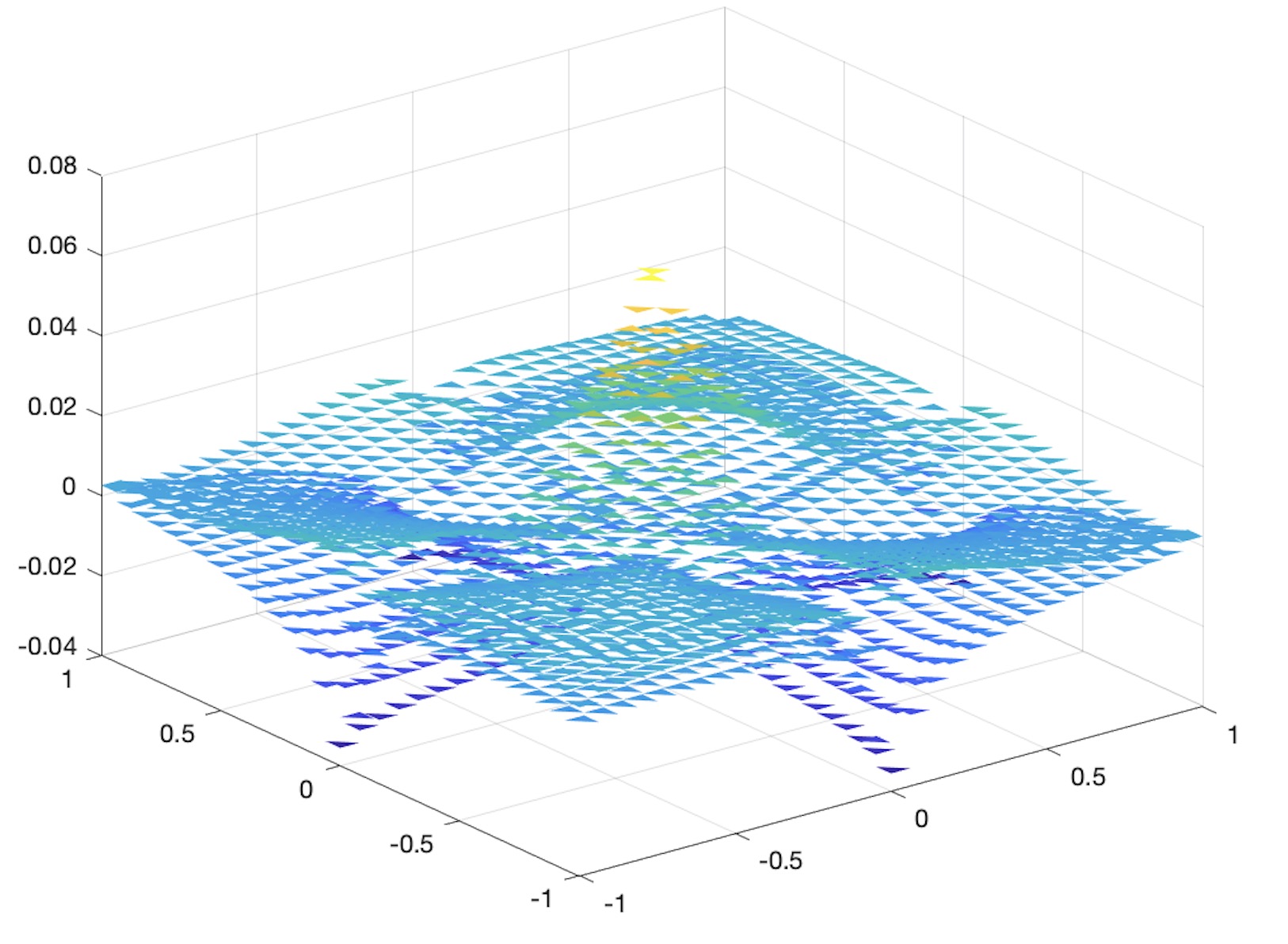}}
\resizebox{2.2in}{2in}{\includegraphics{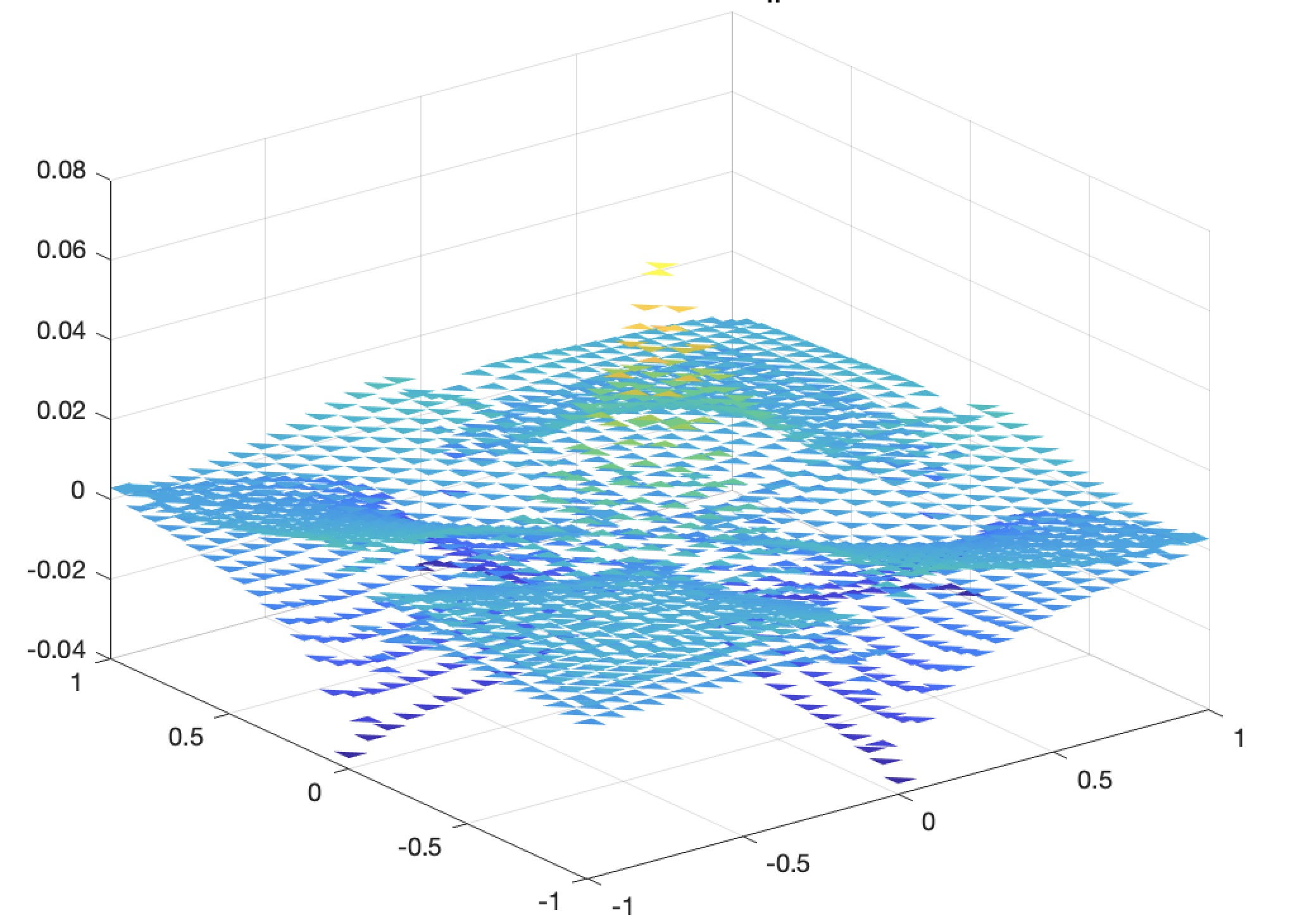}}
\end{tabular}
\caption{Test Case 2: Numerical error for Lagrange multiplier when $C^0$-$P_2(T)/[P_1(\pT)]^2/P_0(T)$ element is applied: left figure is without the term $c(\cdot, \cdot)$ proposed in \cite{wwnondiv}; right figure is with the term $c(\cdot, \cdot)$ proposed in this paper.}
\label{test1-2}
 \end{figure}

\noindent{\bf Test Case 3.} Find $u$ satisfying 
\begin{equation}\label{EQ:NE:800}
\sum_{i, j=1}^2\left(\delta_{ij}+\frac{x_ix_j}{x_1^2+x_2^2}\right)
\partial_{ij}^2 u=f, \qquad \mbox{in} \ \Omega_i \ (i=1, 2).
\end{equation}
For the case of $\alpha>1$, the exact solution $u=|x|^\alpha$ has $H^{1+\alpha-\tau}(\Omega)$ regularity for arbitrarily small $\tau>0$ and the load function is $f=(2\alpha^2-\alpha)|x|^{\alpha-2}$. The Cord\`es condition holds true with $\varepsilon=4/5$.

Tables \ref{NE:TRI:test2-2}-\ref{NE:TRI:test2-6} present the numerical results of the M-PDWG scheme on the domain $\Omega_1=(0,1)^2$. It is clear that the coefficient matrix $a=(a_{ij})_{2\times 2}$ is continuous in the interior of the domain $\Omega_1$, but it fails to be continuous at the corner point $(0,0)$. Note that the exact solution $u=|x|^{1.6}$ has $H^{2.6-\tau}(\Omega)$ regularity for arbitrarily small $\tau>0$.  The numerical approximation indicates that the convergence rates for $\be_g$ and $\gamma_h$ in the discrete $L^2$ norm are of orders ${\cal O}(h^{1.6})$ and ${\cal O}(h^{0.6})$, respectively, which are consist with the theoretical results.  The convergence rate for $e_0$ in the discrete $L^2$ norm seems to be of an order  ${\cal O}(h^{2})$, for which there is no theory available to
apply.

Figures \ref{test2-1}-\ref{test2-2} shows the numerical error $\gamma_h$ for the $C^0$-$P_2(T)/[P_1(\pT)]^2/P_1(T)$ element and the $C^0$-$P_2(T)/[P_1(\pT)]^2/P_0(T)$ element on the domain $\Omega_1$ respectively, compared with the  PDWG scheme proposed in \cite{wwnondiv}.

\begin{table}[h!]
\begin{center}
\caption{Test Case 3: Convergence rates for  $C^0$-
$P_2(T)/[P_1(\pT)]^2/P_1(T)$ element on $\Omega_1$.}\label{NE:TRI:test2-2}
\begin{tabular}{|c|c|c|c|c|c|c|}
\hline
$1/h$        & $\|e_0\|_0 $ & order &  $\|\be_g \|_0 $  & order  &   $\|\gamma_h\|_0$  & order  \\
 \hline
1&	0.06193 &&0.7395&&1.408 &
\\
\hline
2&	0.008210 &	2.915 &	0.1116&	2.729&	0.3570 &	1.980
\\
\hline
4&	0.001760 &	2.222 &	0.04270 &	1.385 &	0.2169 &	0.7190
\\
\hline
8&	4.30E-04	&2.034&	0.01483 &	1.526 &0.1351&	0.6833 
\\
\hline
16&	1.05E-04&	2.035 &	0.005024 &	1.562	&0.08752 &	0.6260
\\
\hline
32&	2.55E-05&	2.042&	0.001681&	1.580&	0.05735 &	0.6098
\\
\hline
\end{tabular}
\end{center}
\end{table}

\begin{table}[h!]
\begin{center}
\caption{Test Case 3: Convergence rates for $C^0$-
$P_2(T)/[P_1(\pT)]^2/P_0(T)$ element on $\Omega_1$.}\label{NE:TRI:test2-6}
\begin{tabular}{|c|c|c|c|c|c|c|}
\hline
$1/h$        & $\|e_0\|_0 $ & order &  $\|\be_g \|_0 $  & order  &   $\|\gamma_h\|_0$  & order  \\
\hline
1	&0.003403	&&	0.4903	&&	0.0650 	&
\\
\hline
2&	0.007769 &	-1.1911 &	0.1774 &	1.467	&0.06253 & 0.05684 
\\
\hline
4&	0.002576 &	1.593&	0.06160 	&1.526	&0.04782 & 0.3870
\\
\hline
8&	7.83E-04&	1.719&	0.02099&	1.554&	0.03270 &	 0.5482
\\
\hline
16&	2.19E-04&	1.839&	0.007048 &	1.574 &	0.02183 &	 0.5832
\\
\hline
32&	5.84E-05	&1.906&	0.002349&	1.585 	&0.01447 & 0.5930 
\\
\hline
\end{tabular}
\end{center}
\end{table}

\begin{figure}[h]
\centering
\begin{tabular}{cc}
\resizebox{2.2in}{2in}{\includegraphics{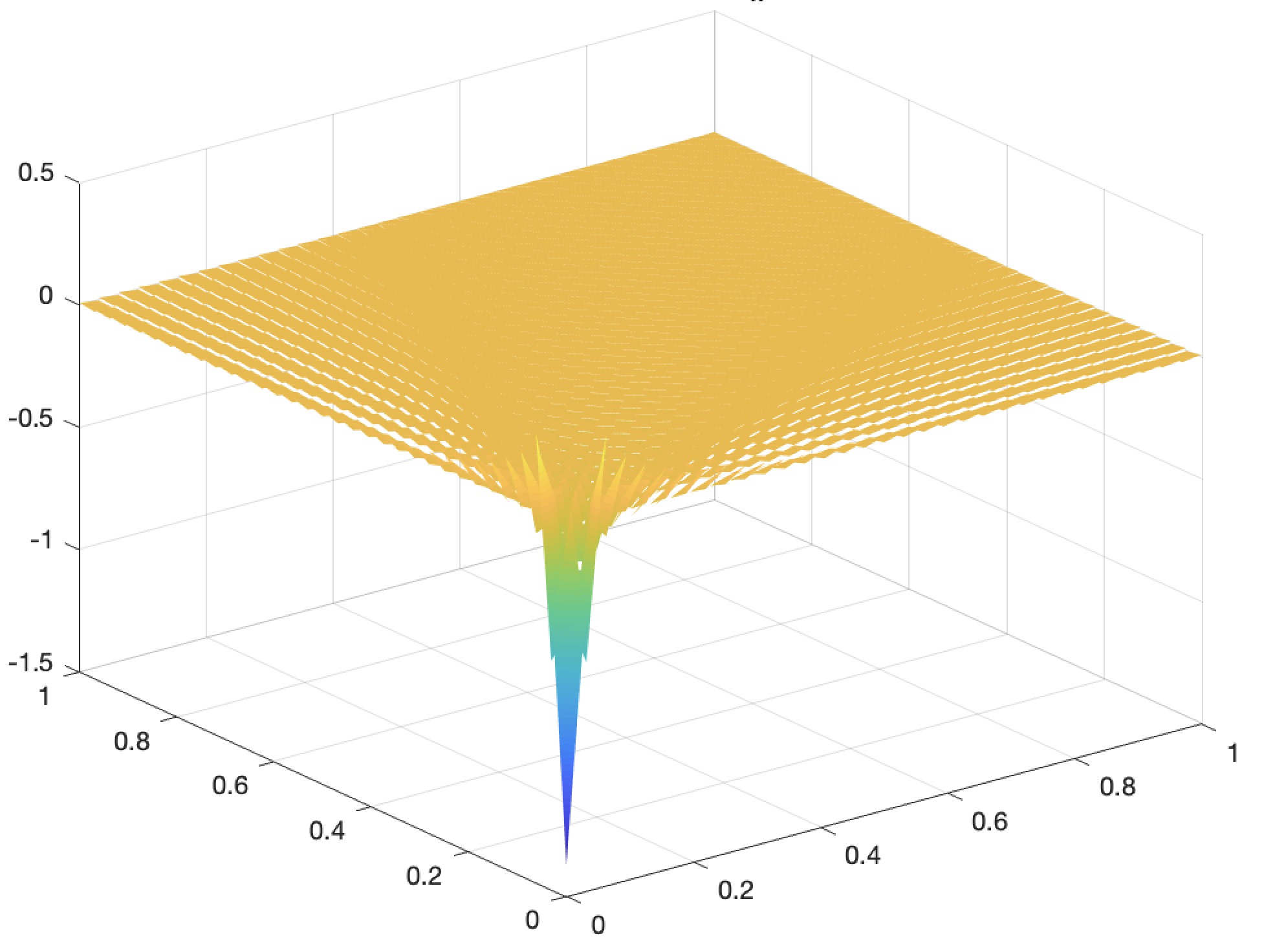}}
\resizebox{2.2in}{2in}{\includegraphics{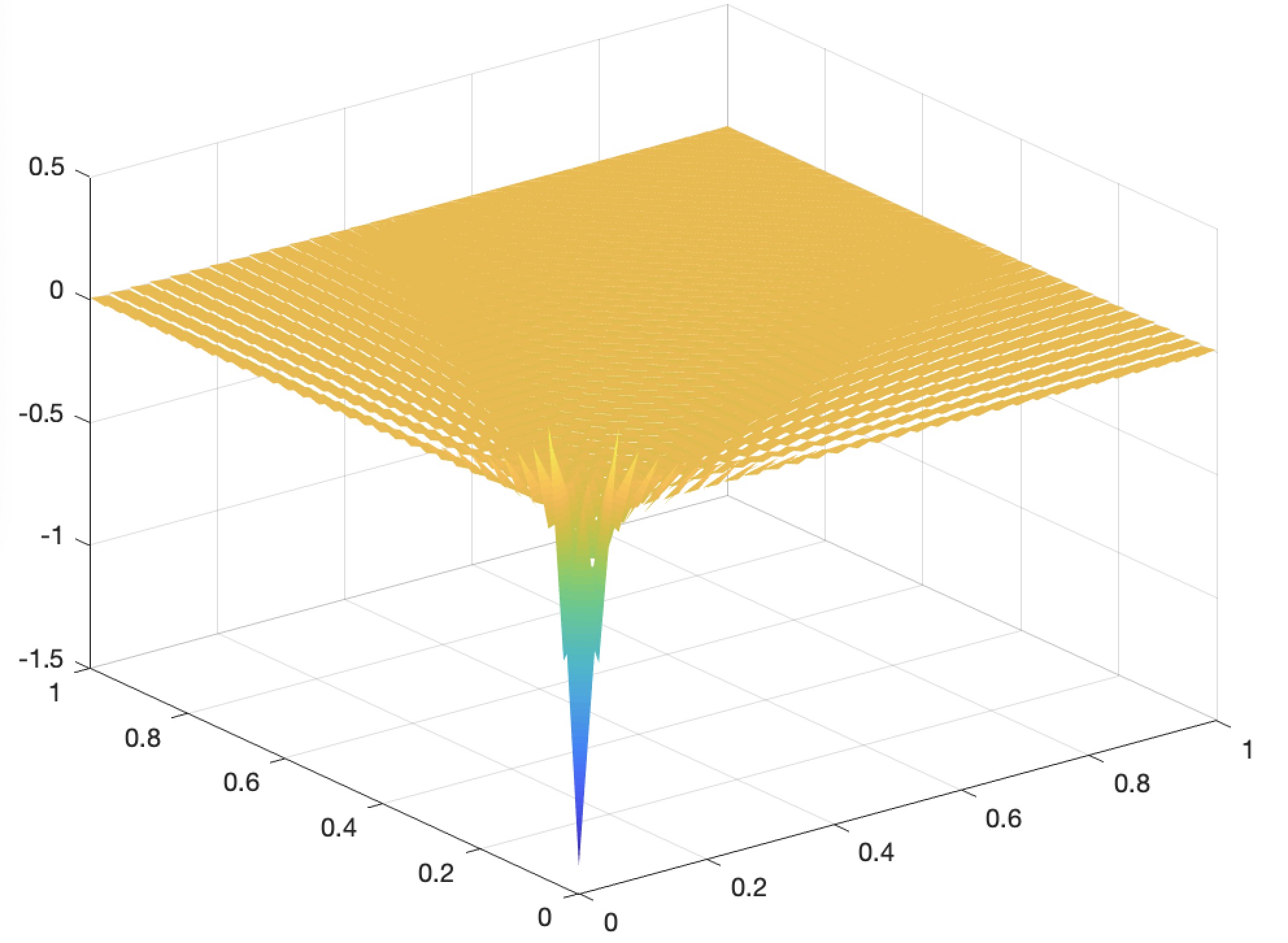}}
\end{tabular}
\caption{Test Case 3: Numerical error for Lagrange multiplier when $C^0$-$P_2(T)/[P_1(\pT)]^2/P_1(T)$ element is applied on $\Omega_1$: left figure is without the term $c(\cdot, \cdot)$ proposed in \cite{wwnondiv}; right figure is with the term $c(\cdot, \cdot)$ proposed in this paper.}
\label{test2-1}
 \end{figure}

\begin{figure}[h]
\centering
\begin{tabular}{cc}
\resizebox{2.2in}{2in}{\includegraphics{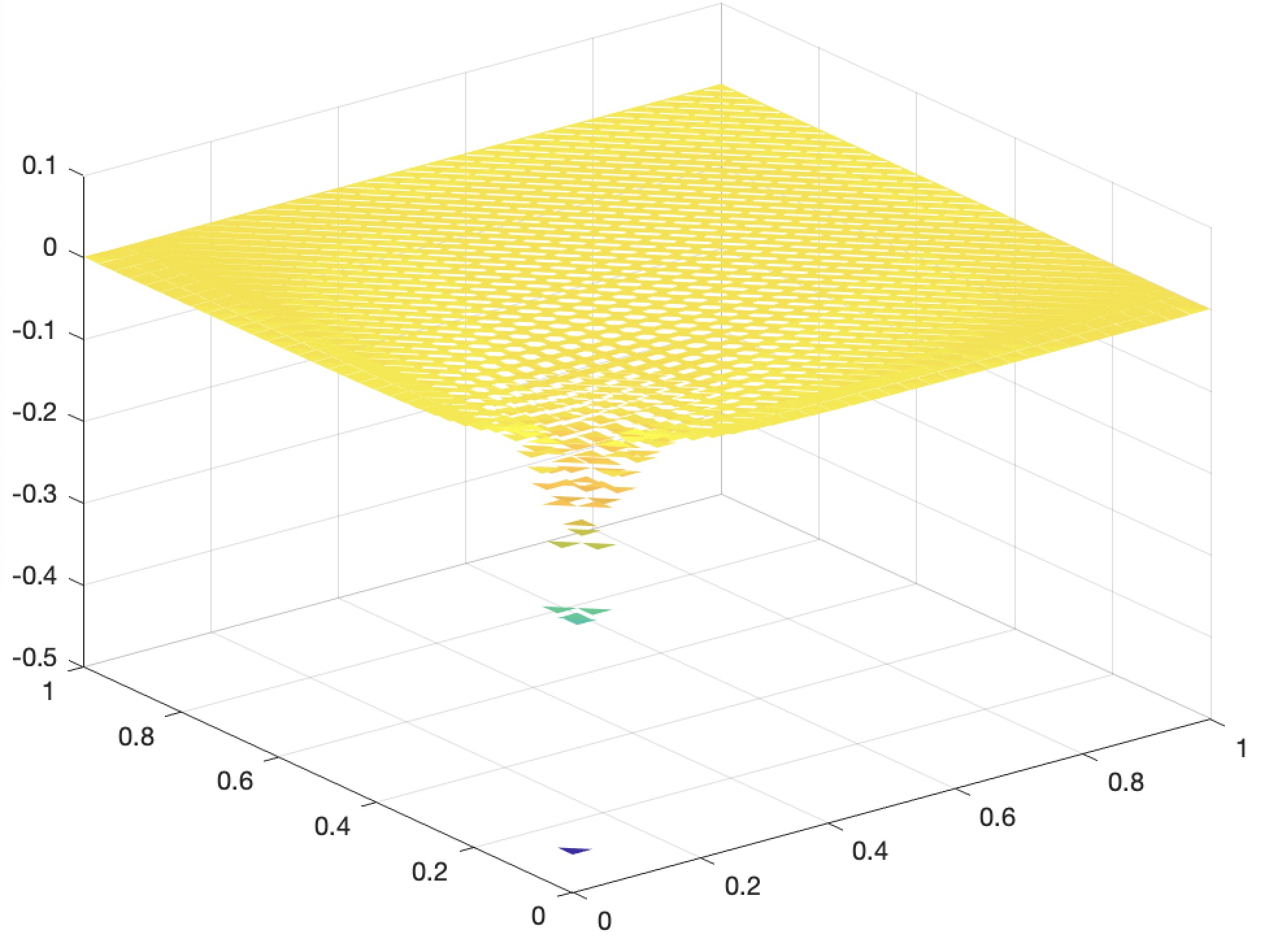}}
\resizebox{2.2in}{2in}{\includegraphics{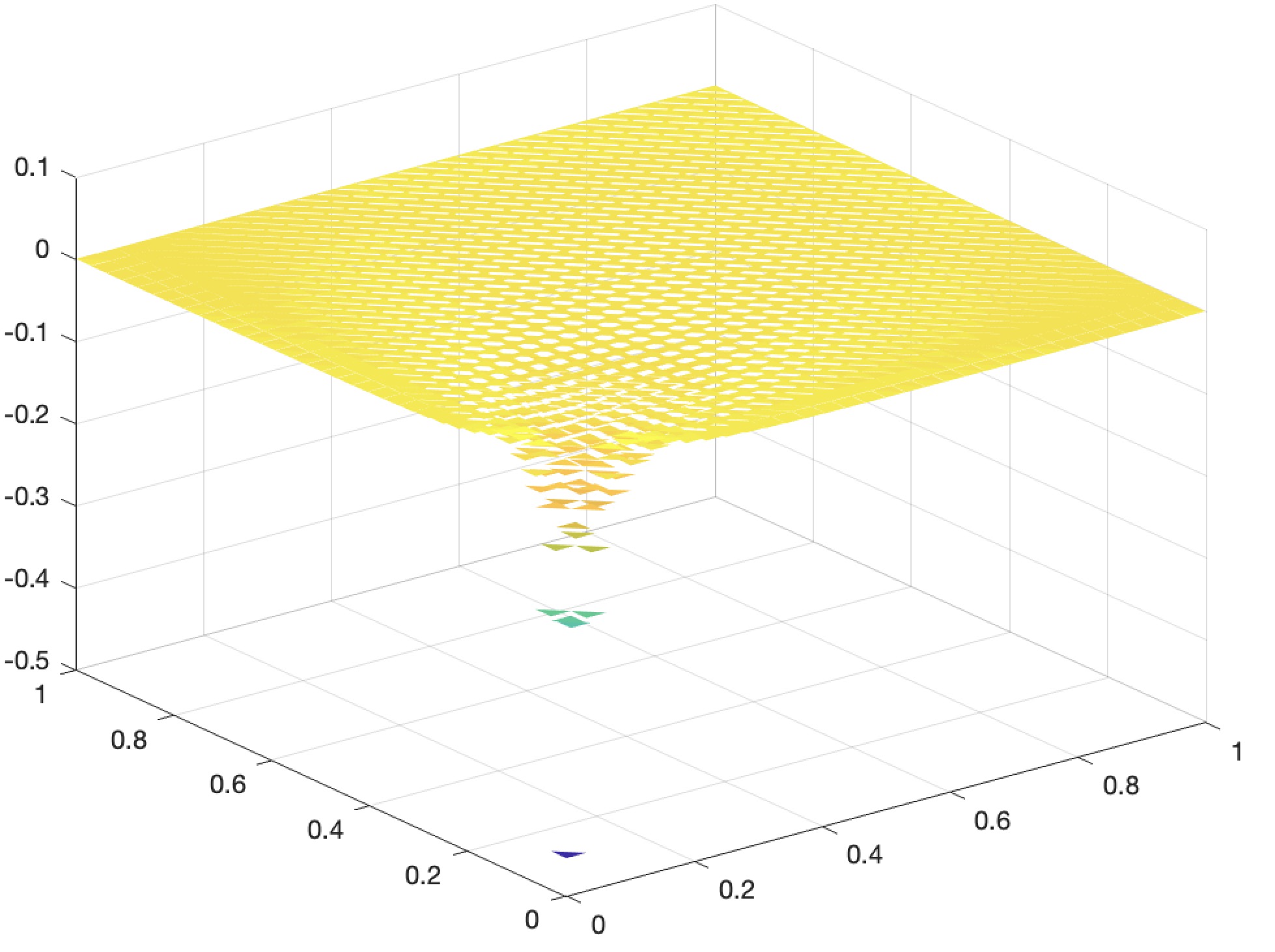}}
\end{tabular}
\caption{Test Case 3: Numerical error for Lagrange multiplier when $C^0$-$P_2(T)/[P_1(\pT)]^2/P_0(T)$ element is applied on $\Omega_1$: left figure is without the term $c(\cdot, \cdot)$ proposed in \cite{wwnondiv}; right figure is with the term $c(\cdot, \cdot)$ proposed in this paper.}
\label{test2-2}
 \end{figure}

Tables \ref{NE:TRI:test3-2}-\ref{NE:TRI:test3-6} demonstrate the
numerical performance of the M-PDWG scheme (\ref{2})-(\ref{32}) for the test equation (\ref{EQ:NE:800}) in the domain $\Omega_2=(-1,1)^2$. The coefficient matrix $a=(a_{ij})_{2\times 2}$ is discontinuous at the center point $(0, 0)$ of the domain $\Omega_2$ so that the duality argument in the convergence theory is not applicable. We observe from Tables \ref{NE:TRI:test3-2}-\ref{NE:TRI:test3-6} that the numerical results are less accurate than the case of $\Omega_1=(0,1)^2$ presented in Tables \ref{NE:TRI:test2-2}-\ref{NE:TRI:test2-6}. The convergence rate for $\gamma_h$ in the $L^2$ norm is of an order ${\cal O}(h^{0.6})$, which is consistent with the theory; while the convergence rates for $e_0$ and $\be_g$ in the $L^2$ norm are both of an order ${\cal O}(h)$ or slightly higher.

Figures \ref{test2-3}-\ref{test2-4} shows the numerical error $\gamma_h$ for the $C^0$-$P_2(T)/[P_1(\pT)]^2/P_1(T)$ element and the $C^0$-$P_2(T)/[P_1(\pT)]^2/P_0(T)$ element on the domain $\Omega_2$ respectively, compared with the PDWG scheme proposed in \cite{wwnondiv}.

\begin{table}[h!]
\begin{center}
\caption{Test Case 3: Convergence rates for $C^0$-
$P_2(T)/[P_1(\pT)]^2/P_1(T)$ element on $\Omega_2$.}\label{NE:TRI:test3-2}
\begin{tabular}{|c|c|c|c|c|c|c|}
\hline
$2/h$        & $\|e_0\|_0 $ & order &  $\|\be_g \|_0 $  & order  &   $\|\gamma_h\|_0$  & order  \\
 \hline
1	&0.8998	&&	1.207&&	0.4146&
\\
\hline
2	&0.7142&	0.3333&	1.808&	-0.5834 &	2.289	&-2.465
\\
\hline
4&	0.1928 &	1.889&	1.244&	0.5394 	&4.685 	&-1.034
\\
\hline
8&	0.04503 &	2.098 &	0.0967 &	3.685	&0.5329&	3.136 
\\
\hline
16&	0.02497 &	0.8506&	0.05352&	0.8540 &0.3078	&0.7919
\\
\hline
32&	0.01242 &	1.007 &	0.02806&	0.9316&0.1958&	0.6526
\\
\hline
\end{tabular}
\end{center}
\end{table}

\begin{table}[H]
\begin{center}
\caption{Test Case 3: Convergence rates for $C^0$-
$P_2(T)/[P_1(\pT)]^2/P_0(T)$ element on $\Omega_2$.}\label{NE:TRI:test3-6}
\begin{tabular}{|c|c|c|c|c|c|c|}
\hline
$2/h$        & $\|e_0\|_0 $ & order &  $\|\be_g \|_0 $  & order  &   $\|\gamma_h\|_0$  & order  \\
\hline
1	&6.82E-01	&&	0.5800 	&&	0.1091 	&
\\
\hline
2&	6.13E-01&	0.1518&	0.7084&	-0.2884 &	0.08120 &	0.4271
\\
\hline
4&	2.54E-01&	1.273&	0.4067 &	0.8004 &	0.05057 &	0.6831 
\\
\hline
8&	1.12E-01&	1.175&	0.2177&	0.9018 &	0.04179&	0.2753
\\
\hline
16&	5.12E-02&	1.137&	0.1101 &	0.9829	&0.02969 &	0.4930
\\
\hline
32	&0.02354&	1.120&	0.05402 &	1.028	&0.02011 	& 0.5620
\\
\hline
\end{tabular}
\end{center}
\end{table}

\begin{figure}[h]
\centering
\begin{tabular}{cc}
\resizebox{2.2in}{2in}{\includegraphics{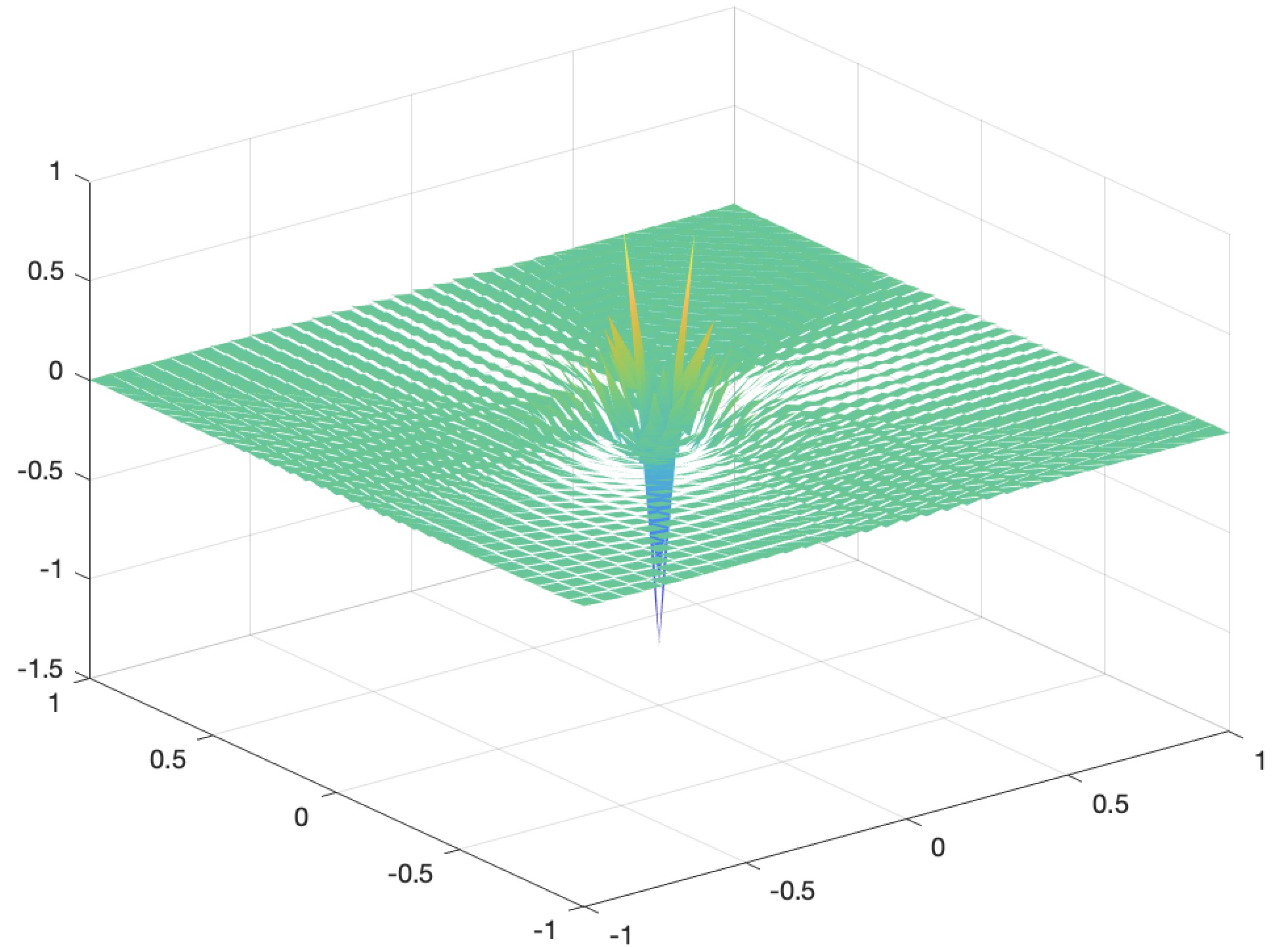}}
\resizebox{2.2in}{2in}{\includegraphics{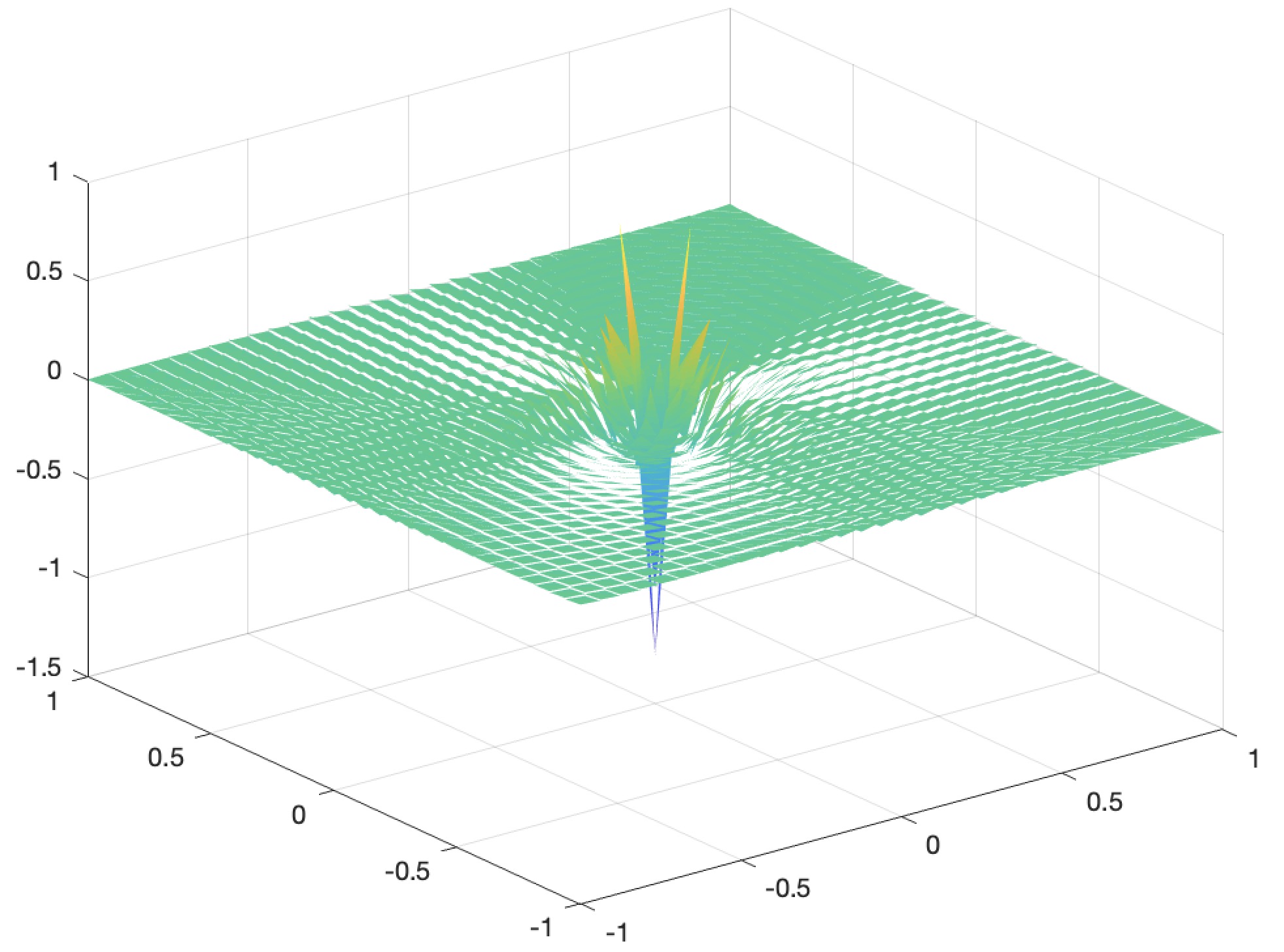}}
\end{tabular}
\caption{Test Case 3: Numerical error for Lagrange multiplier when $C^0$-$P_2(T)/[P_1(\pT)]^2/P_1(T)$ element is applied on $\Omega_2$: left figure is without the term $c(\cdot, \cdot)$ proposed in \cite{wwnondiv}; right figure is with the term $c(\cdot, \cdot)$ proposed in this paper.}
\label{test2-3}
 \end{figure}

\begin{figure}[h]
\centering
\begin{tabular}{cc}
\resizebox{2.2in}{2in}{\includegraphics{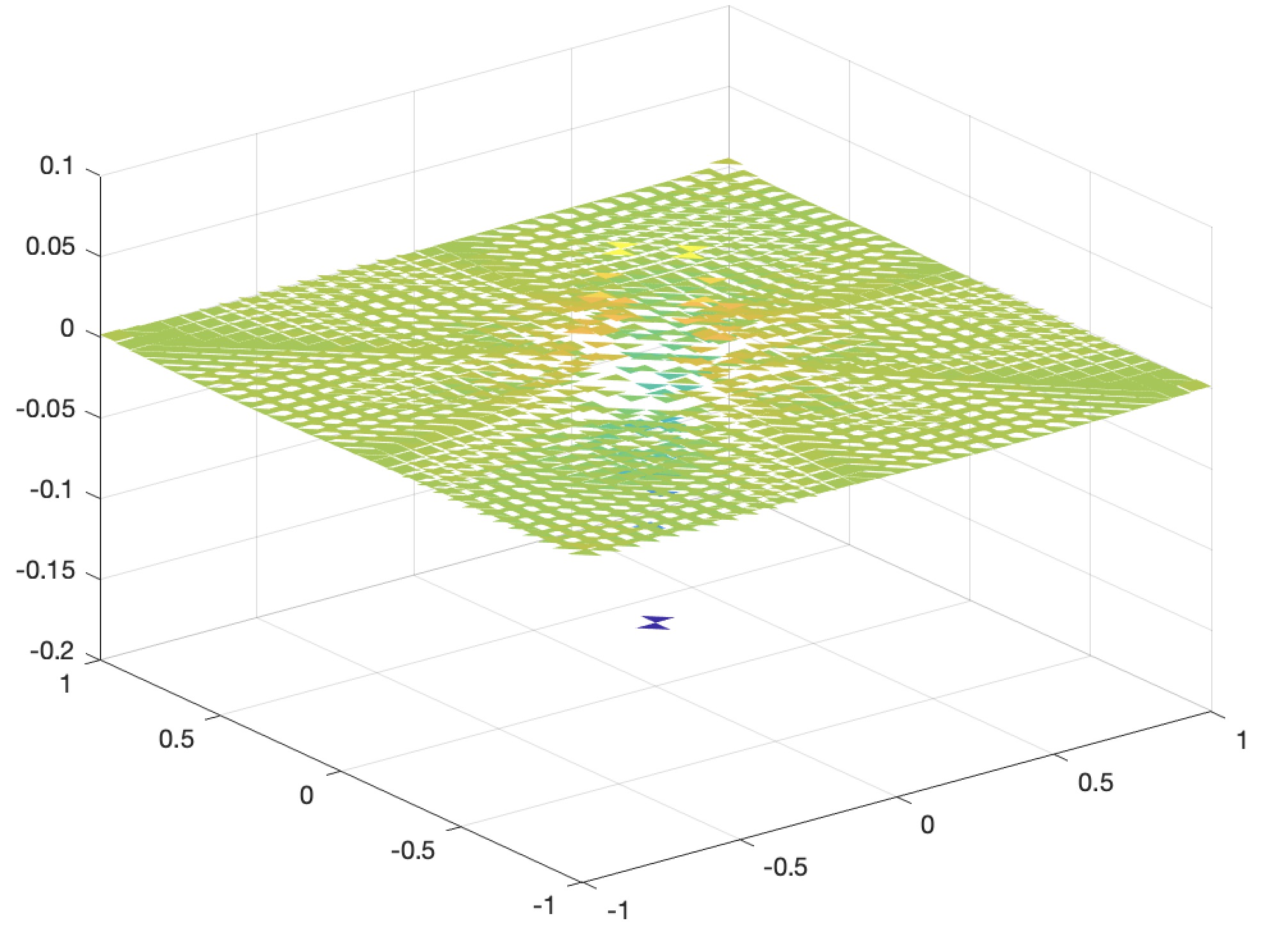}}
\resizebox{2.2in}{2in}{\includegraphics{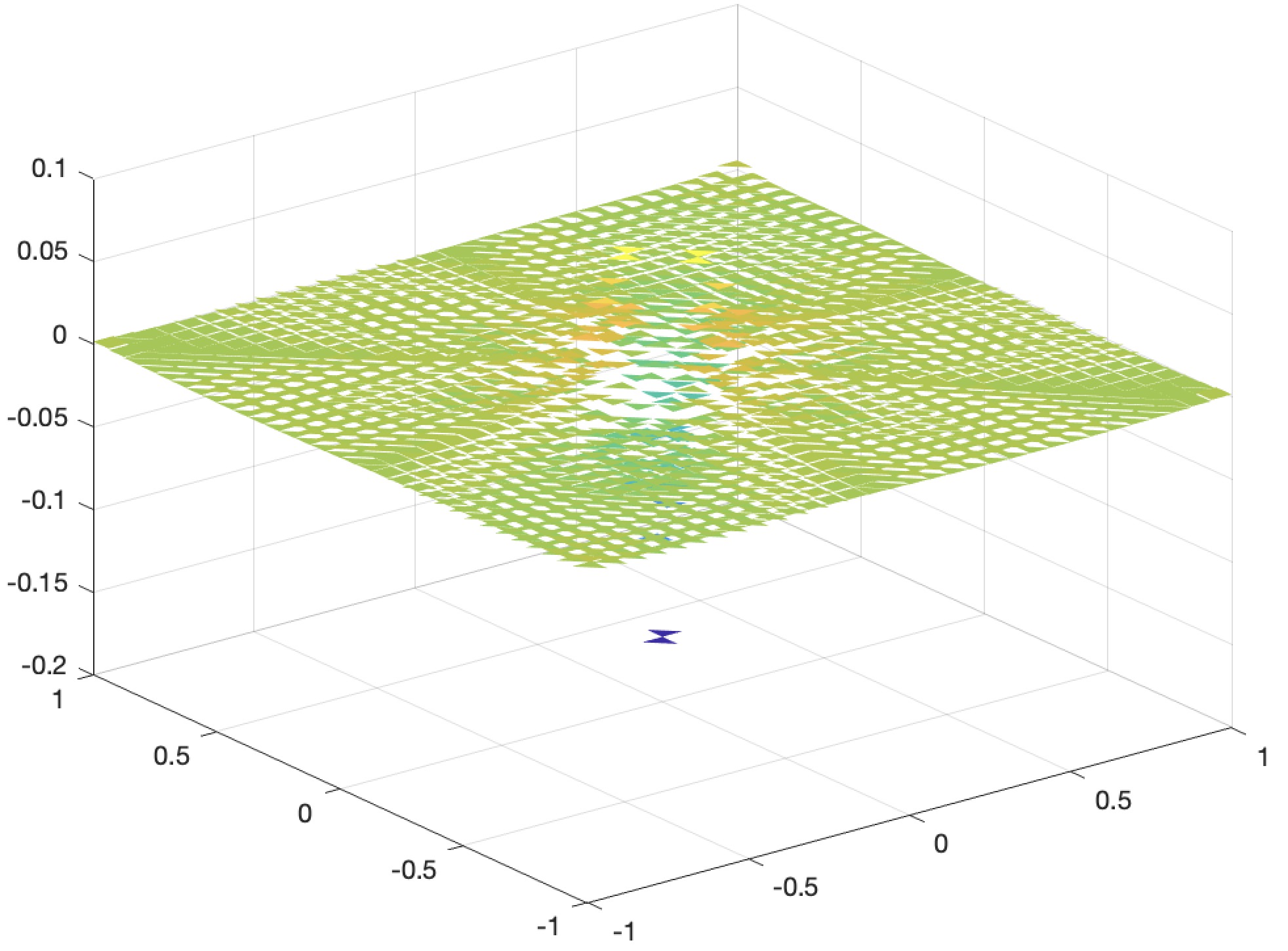}}
\end{tabular}
\caption{Test Case 3: Numerical error for Lagrange multiplier when $C^0$-$P_2(T)/[P_1(\pT)]^2/P_0(T)$ element is applied on $\Omega_2$: left figure is without the term $c(\cdot, \cdot)$ proposed in \cite{wwnondiv}; right figure is with the term $c(\cdot, \cdot)$ proposed in this paper.}
\label{test2-4}
 \end{figure}
 
 \section*{Acknowledgement}
 I would like to express my gratitude to  Dr. Junping Wang for his valuable discussion and suggestions.

\end{document}